\documentclass[preprint,12pt]{elsarticle}

\usepackage[a4paper,top=2cm,bottom=2cm,left=1.9cm,right=1.9cm]{geometry}
\usepackage{amsmath,amssymb,amsthm,mathtools,bm,stmaryrd}
\usepackage{booktabs,tabularx,array,float}
\usepackage{graphicx}
\usepackage{siunitx}
\usepackage{xcolor}
\usepackage{microtype}
\usepackage{hyperref}
\hypersetup{hidelinks}
\usepackage[nameinlink,capitalise]{cleveref}

\journal{Journal of Computational Physics}

\newtheorem{theorem}{Theorem}

\newtheorem{remark}{Remark}

\newcommand{\e}[1]{\times10^{#1}}
\newcommand{\state}{\bm q}
\newcommand{\Fa}{\bm F^{a}}
\newcommand{\Fv}{\bm F^{v}}
\newcommand{\sourceib}{\bm s_{\mathrm{IB}}}
\newcommand{\vel}{\bm u}
\newcommand{\vels}{\bm u_s}
\newcommand{\vrel}{\bm v}
\newcommand{\entropyvar}{\bm w}
\newcommand{\x}{\bm x}
\newcommand{\R}{\mathbb{R}}
\newcommand{\jump}[1]{\llbracket #1\rrbracket}
\newcommand{\avg}[1]{\{\!\{#1\}\!\}}
\newcommand{\inner}[2]{\left\langle #1,#2\right\rangle_{E,P}}
\newcommand{\dd}{\,\mathrm{d}}

\begin{document}

\begin{frontmatter}


\title{Simultaneous energy and entropy stability of volume-penalisation immersed-boundary methods for the compressible Navier--Stokes equations}

\author[1]{Abbas Ballout\corref{cor1}}
\ead{abbas.ballout@upm.es}
\author[1]{Albert Jiménez-Ramos}
\author[1,2]{Gonzalo Rubio}
\author[1,2]{Esteban Ferrer}
\cortext[cor1]{Corresponding author.}

\affiliation[1]{organization={ETSIAE-UPM-School of Aeronautics, Universidad Politécnica de Madrid},
            addressline={Plaza Cardenal Cisneros 3},
            postcode={28040},
            city={Madrid},
            country={Spain}}

\affiliation[2]{organization={Center for Computational Simulation, Universidad Politécnica de Madrid},
            addressline={Campus de Montegancedo, Boadilla del Monte},
            postcode={28660},
            city={Madrid},
            country={Spain}}

\begin{abstract}
Immersed-boundary methods avoid body-fitted meshes, but the energy or entropy stability of the underlying flow discretisation does not automatically extend to the immersed treatment. This manuscript analyses the volume-penalised compressible Navier--Stokes equations for both energy and entropy stability, on the continuous and semi-discrete levels, and integrates the isolated source exactly in time. The source is formulated as a one-parameter family indexed by $\theta$, whose momentum component relaxes the fluid velocity towards the prescribed solid velocity; $\theta$ sets how much of the resulting kinetic-energy loss leaves the domain rather than being converted to internal energy. The two stability requirements impose opposing bounds on $\theta$, with dissipation of the relative total energy requiring $\theta\geq0$ and stability of the mathematical entropy requiring $\theta\leq1$. Where penalisation acts, both conditions are therefore satisfied if and only if $0\leq\theta\leq1$, and the symmetric value $\theta=1/2$ divides the kinetic-energy loss equally between the domain and the internal energy. We conclude that any entropy-stable split-form discontinuous Galerkin spectral element method remains semi-discretely entropy stable once the source is added at the collocation nodes, for stationary and prescribed moving bodies alike.

Numerical experiments with a split-form DGSEM confirm the analysis. A homogeneous relaxation reproduces the closed-form energy and entropy changes of the exact integrated source to within $2.4\e{-11}$. For a Taylor--Green vortex with an immersed sphere and for a translating cylinder, the discrete energy and entropy budgets follow the predicted dependence on $\theta$: the immersed entropy contribution vanishes to machine precision at $\theta=1$, and the relative total energy is conserved to round-off at $\theta=0$. For the moving body, the total energy grows for every admissible $\theta$ while the relative total energy does not, confirming the latter as the appropriate stability measure.
\end{abstract}

\begin{keyword}
immersed boundary \sep volume penalisation \sep entropy stability \sep energy stability \sep split form \sep discontinuous Galerkin spectral element method \sep compressible Navier--Stokes
\end{keyword}

\end{frontmatter}

\tableofcontents

\section{Introduction and state of the art}
\label{sec:introduction}

Immersed-boundary (IB), embedded-boundary, and fictitious-domain techniques replace a body-fitted mesh by a geometrically simpler background grid on which the influence of an embedded solid is imposed through forcing, reconstruction, modified quadrature, weak constraints, or cut-cell fluxes. Since the original immersed-boundary formulation of Peskin, these methods have developed into a broad family encompassing diffuse-interface forcing, direct-forcing and ghost-cell approaches, cut-cell methods, fictitious-domain formulations, shifted-boundary methods, and volume penalisation \cite{Peskin2002,MittalIaccarino2005,KimChoi2019,Mittal2023,Colombo2025}. Their principal attraction is geometric flexibility, since complex or moving bodies can be represented without repeated generation of boundary-conforming meshes.

This flexibility does not, however, imply a discrete energy or entropy estimate. Even when the bulk Navier--Stokes operator is energy stable or entropy stable, the immersed treatment adds algebraic components that must be included in the stability analysis. In classical Eulerian--Lagrangian IB methods, the interpolation and force-spreading operators, elastic response, and temporal coupling determine the stability boundary. Explicit, approximate-implicit, and implicit treatments can exhibit substantially different stability behaviour \cite{TuPeskin1992}, and dedicated Fourier analysis is required even for planar model problems \cite{HuaPeskin2022}. Energy-stable variants therefore rely on specially designed couplings rather than following automatically from the standard IB formulation \cite{HuLai2013}. In fictitious-domain methods, stability depends on the multiplier or interpolation coupling, and energy estimates have been established for specific one-field formulations \cite{WangJimackWalkley2019}. Likewise, unfitted Nitsche and CutFEM methods require stabilisation to control arbitrarily small intersections and conditioning; ghost penalties or related terms are introduced to recover coercivity independently of how the interface cuts the mesh \cite{BurmanHansbo2012,BurmanHansbo2014}.

Cut-cell methods illustrate the same point in a conservative setting. Small cut cells impose severe explicit time-step restrictions, and high-order energy or entropy stability requires compatible quadrature, summation-by-parts (SBP) operators, numerical fluxes, and stabilisation. Recent work has proved $L^2$ energy stability for a high-order cut-cell discontinuous Galerkin method combined with state redistribution \cite{TaylorWilcoxChan2025}. A general high-order entropy-stable DG construction on cut meshes requires positive cut-cell quadrature and a skew-hybridised SBP formulation \cite{TaylorChan2026}. These results demonstrate that conservation or geometric consistency alone is insufficient, and stability must be designed and proved for the complete immersed discretisation.

Volume penalisation follows a different route. The solid is represented as a porous region of vanishing permeability and the boundary condition is imposed through a volumetric relaxation source \cite{Angot1999}. For stationary Brinkman drag, contraction with the velocity immediately produces a non-positive pointwise contribution to kinetic energy. This algebraic structure makes energy control more direct than in immersed formulations whose stability depends on embedded-surface interpolation, cut-element quadrature, or multiplier coupling. Volume penalisation also preserves the tensor-product structure of the background discretisation and has been coupled to high-order flux-reconstruction and discontinuous Galerkin methods for stationary and moving geometries \cite{KouEtAl2022VP,KouEtAl2022Eigen}. Subsequent work analysed its modified equation and the dependence of the numerical dissipation on the penalisation parameter \cite{LlorenteEtAl2023}, and combined the penalisation with selective-frequency damping to reduce oscillations inside the solid \cite{KouFerrer2023JCP,KouFerrer2023PoF}. Compressible implementations and extensive flow tests have been reported in the high-order solver HORSES3D \cite{FerrerEtAl2023} and in the aeronautical solver CODA \cite{NunezEtAl2025}. Colombo et al. presented a complementary high-order surrogate-boundary formulation and a discussion of alternative IB families. \cite{Colombo2025}.

Previous analyses of volume penalisation have established convergence, eigensolution properties, and additional damping in model equations \cite{Angot1999,KouEtAl2022Eigen,LlorenteEtAl2023}. What has not been made explicit is the thermodynamic structure of a compressible momentum--energy penalty that relaxes the fluid velocity towards a prescribed solid velocity. In particular, one must distinguish kinetic energy, total energy, energy measured relative to a moving body, and entropy stability in the sense of a convex mathematical entropy. For a moving body, kinetic or total energy need not decrease because the prescribed motion may perform physical work on the fluid. The appropriate estimate separates this work from dissipation of energy relative to the solid velocity.

\subsection{Energy and entropy stability in the present context}
\label{sec:intro_stability}

Energy stability implies that a continuous or discrete formulation satisfies a bound for a non-negative energy functional, up to prescribed boundary work and external forcing. In the context of this work, the relevant functional is the total energy. For a moving body, total energy is generally not monotone because the body can inject mechanical power. The useful statement is instead an estimate for total energy measured relative to the prescribed solid velocity, with the work associated with acceleration, rotation, or deformation retained explicitly.

Entropy stability concerns a convex mathematical entropy $\eta(\state)$, its entropy variables $\entropyvar=\partial\eta/\partial\state$, an associated entropy flux, and a semi-discrete inequality that mirrors the continuous second-law structure \cite{Tadmor2003,FisherCarpenter2013,Chandrashekar2013,CarpenterEtAl2014,GassnerEtAl2018}. Entropy-conservative two-point fluxes provide the volume building block, but a complete scheme also requires entropy-stable interface dissipation, an entropy-stable viscous operator, entropy-compatible physical boundaries, and a suitable time discretisation. Once an immersed source is added, the bulk estimate survives only if the contraction of that source with the entropy variables is non-positive.

For the penalisation family considered here, the two requirements constrain the same quantity and pull in opposite directions. Both are governed by how the work done by the penalisation is divided between kinetic and internal energy, and that division is what the parameter $\theta$ of \cref{sec:source_family} controls. We show that the relative-energy estimate bounds $\theta$ from below and that the entropy estimate bounds it from above, so that simultaneous stability holds on the intersection of the two bounds. Establishing that interval, at the continuous level and for the split-form discretisation, occupies the remainder of the paper.

The main contributions of this paper are:
\begin{itemize}
    \item a one-parameter family of compressible volume-penalisation sources that contains the two members already in use as the end point $\theta=1$ and the symmetric $\theta=1/2$;
    \item a relative-energy balance for the complete system with a prescribed moving body, in which the penalisation is the only non-transport contribution for a stationary body or a uniform translation;
    \item a condition on $\theta$ for simultaneous relative-energy and entropy stability, on the admissible interval $0\leq\theta\leq1$;
    \item a split-form DGSEM formulation and semi-discrete entropy and relative-energy theorems for the coupled Navier--Stokes--penalisation system, including the result that an arbitrary subset of collocation nodes may be penalised without losing the source entropy estimate;
    \item the consequence that the source is modular, any bulk discretisation already satisfying an entropy or energy estimate retaining it once the source is added;
    \item a closed-form integration of the isolated source over a time step;
    \item numerical verification in a high-order solver of the predicted source rates, of the immersed contribution to the semi-discrete entropy inequality, and of the relative-energy balance for a moving body.
\end{itemize}

The paper is organised as follows. \Cref{sec:governing} introduces the compressible Navier--Stokes equations and the source family. \Cref{sec:continuous} derives the continuous entropy and relative-energy balances. \Cref{sec:dgsem_formulation,sec:semidiscrete} present the split-form DGSEM and its semi-discrete stability proof. \Cref{sec:time} integrates the isolated source in closed form over a time step, and \cref{sec:numerical,sec:conclusions} present the numerical experiments and conclusions.

\section{Compressible Navier--Stokes equations with volume penalisation}
\label{sec:governing}

Let $\Omega\subset\R^3$ be a fixed background domain with boundary $\partial\Omega$, and let $t\in(0,t_f]$. Scalars are written in standard type, vectors and tensors in boldface. The conservative state is
\begin{equation}
    \state=
    \begin{bmatrix}
        \rho\\
        \rho\vel\\
        \rho E
    \end{bmatrix},
    \qquad
    \rho E=\rho e+\frac12\rho\lvert\vel\rvert^2,
    \label{eq:state}
\end{equation}
where $\rho$ is density, $\vel$ velocity, and $e$ and $E$ are the specific internal and specific total energies, so that $\rho e$ and $\rho E$ are the internal and total energies per unit volume. The compressible Navier--Stokes equations with an immersed source are
\begin{equation}
    \partial_t\state+\nabla\cdot\Fa(\state)
    =\nabla\cdot\Fv(\state,\nabla\state)+\sourceib(\state,\x,t).
    \label{eq:compact_ns}
\end{equation}
For Cartesian direction $j=1,2,3$, the inviscid and viscous fluxes are
\begin{equation}
    \Fa_j(\state)=
    \begin{bmatrix}
        \rho u_j\\
        \rho\vel u_j+p\bm e_j\\
        (\rho E+p)u_j
    \end{bmatrix},
    \qquad
    \Fv_j(\state,\nabla\state)=
    \begin{bmatrix}
        0\\
        \bm\tau\bm e_j\\
        \vel\cdot(\bm\tau\bm e_j)-j_{h,j}
    \end{bmatrix}.
    \label{eq:fluxes}
\end{equation}
Here $\bm\tau$ is the viscous stress and $\bm j_h$ is the heat-flux vector. For a Newtonian calorically perfect gas,
\begin{equation}
    \bm\tau
    =\mu\left(\nabla\vel+\nabla\vel^T
    -\frac{2}{3}(\nabla\cdot\vel)\bm I\right),
    \qquad
    \bm j_h=-\kappa\nabla T,
    \label{eq:constitutive}
\end{equation}
with $\mu\geq0$, $\kappa\geq0$, and
\begin{equation}
    p=(\gamma-1)\rho e,
    \qquad
    \rho>0,
    \qquad
    p>0.
    \label{eq:eos}
\end{equation}
Appropriate initial and physical boundary conditions complete the problem.

\subsection{The $\theta$ family of momentum--energy penalisation}
\label{sec:source_family}

Let $\chi(\x,t)\in[0,1]$ denote a sharp or smooth immersed-body mask and let
\begin{equation}
    \lambda(\x,t)=\frac{\chi(\x,t)}{\tau_p},
    \qquad \tau_p>0,
    \label{eq:lambda}
\end{equation}
be the non-negative relaxation rate, $\tau_p$ being the penalisation time scale. The prescribed body velocity is $\vels(\x,t)$ and the relative velocity is
\begin{equation}
    \vrel=\vel-\vels.
    \label{eq:relative_velocity}
\end{equation}
The momentum source is the Brinkman relaxation
\begin{equation}
    \bm s_{\rho\vel}=-\lambda\rho\vrel.
    \label{eq:momentum_source}
\end{equation}
We accompany it with a one-parameter energy source
\begin{equation}
    s_E^{(\theta)}
    =\left(\vels+\theta\vrel\right)\cdot\bm s_{\rho\vel}
    =-\lambda\rho\left(\vels\cdot\vrel+\theta\lvert\vrel\rvert^2\right),
    \qquad \theta\in\R.
    \label{eq:energy_source}
\end{equation}
The complete source is therefore
\begin{equation}
    \sourceib^{(\theta)}=
    \begin{bmatrix}
        0\\
        -\lambda\rho\vrel\\
        -\lambda\rho\left(\vels\cdot\vrel+\theta\lvert\vrel\rvert^2\right)
    \end{bmatrix}.
    \label{eq:source_family}
\end{equation}
The symmetric member $\theta=1/2$ can be written as
\begin{equation}
    s_E^{(1/2)}
    =\frac{\lambda\rho}{2}
    \left(\lvert\vels\rvert^2-\lvert\vel\rvert^2\right).
    \label{eq:symmetric_source}
\end{equation}
Two members of the $\theta$ family are already in use. The endpoint $\theta=1$ is the adiabatic-wall penalisation of Abgrall et al.\ \cite{AbgrallEtAl2014}, while for an isothermal wall they penalise the temperature in addition, which lies outside \cref{eq:source_family}. The symmetric member \cref{eq:symmetric_source} was adopted in its place by Kou and co-authors, and has since been implemented and tested in high-order flux-reconstruction and DG frameworks for stationary and moving geometries, including the HORSES3D and CODA solvers \cite{KouEtAl2022VP,KouFerrer2023PoF,FerrerEtAl2023,NunezEtAl2025}. These references demonstrate practical use and show that \cref{eq:source_family} is defined independently of any particular code. Other models lie outside the family. The work of Liu and Vasilyev \cite{LiuVasilyev2007} for instance, relaxes the temperature towards that of the solid and rescales the mass flux by a porosity inside the body. The work here concerns \cref{eq:source_family} alone.

\subsection{Interpretation of the energy source}
\label{sec:source_interpretation}

Take the body at rest for the moment, $\vels=\bm0$, so that $\vrel=\vel$. The momentum source is then $\bm s_{\rho\vel}=-\lambda\rho\vel$ and destroys kinetic energy at the rate $\lambda\rho\lvert\vel\rvert^2$, which carries no $\theta$.

\cref{eq:energy_source} through $\theta$ states where that energy goes. At rest it reduces to
\begin{equation}
    s_E^{(\theta)}=-\theta\lambda\rho\lvert\vel\rvert^2 ,
    \label{eq:energy_source_static}
\end{equation}
so $\theta$ is the fraction of the dissipated kinetic energy that leaves the system. The remainder has nowhere else to go but through the internal energy. Since $\rho e=\rho E-\tfrac12\rho\lvert\vel\rvert^2$, the internal energy rises at the rate $(1-\theta)\lambda\rho\lvert\vel\rvert^2$.

In particular, at $\theta=0$ the energy source vanishes, so the total energy is unchanged and the whole kinetic loss becomes heat. In contrast, at $\theta=1$ the energy source removes the kinetic loss in full from the total energy and the internal energy is unchanged. Intermediate values divide it in the ratio $1-\theta$ to $\theta$.

\section{Continuous entropy and relative-energy analysis}
\label{sec:continuous}

This section derives the two continuous estimates on which the rest of the paper rests, and the bounds they place on $\theta$. We first restate the entropy structure of the source-free compressible Navier--Stokes equations, the entropy pair, all of which is classical and is reproduced only so that the immersed source can be placed exactly within it. We then proceed to energy, and show that the total energy carries no definite sign once the body moves, unless measured relative to the solid velocity.

\subsection{Entropy pair, entropy variables, entropy potential and the entropy contribution of the immersed source}
\label{sec:entropy_pair}

Consider the general viscous conservation law with source in \cref{eq:compact_ns}. An entropy pair consists of a strictly convex scalar entropy $\eta(\state)$ and entropy fluxes $f_{\eta,j}(\state)$ satisfying the compatibility relation
\begin{equation}
    \frac{\partial f_{\eta,j}}{\partial\state}
    =\entropyvar^T\frac{\partial\Fa_j}{\partial\state},
    \qquad
    \entropyvar=\frac{\partial\eta}{\partial\state},
    \qquad j=1,2,3.
    \label{eq:entropy_compatibility}
\end{equation}
The corresponding entropy potential is, by definition
\begin{equation}
    \psi_j=\entropyvar^T\Fa_j-f_{\eta,j}.
    \label{eq:entropy_potential_definition}
\end{equation}
For the calorically perfect gas, define the dimensionless thermodynamic entropy
\begin{equation}
    s=\ln\left(\frac{p}{p_{\mathrm{ref}}}\right)
      -\gamma\ln\left(\frac{\rho}{\rho_{\mathrm{ref}}}\right),
    \label{eq:dimensionless_entropy}
\end{equation}
where the positive reference values only set an additive entropy constant, and define the convex mathematical entropy
\begin{equation}
    \eta(\state)=-\frac{\rho s}{\gamma-1}.
    \label{eq:mathematical_entropy}
\end{equation}
The entropy variables are
\begin{equation}
    \entropyvar=
    \begin{bmatrix}
        \displaystyle\frac{\gamma-s}{\gamma-1}
        -\frac{\rho\lvert\vel\rvert^2}{2p}\\[2mm]
        \displaystyle\frac{\rho\vel}{p}\\[2mm]
        \displaystyle-\frac{\rho}{p}
    \end{bmatrix}.
    \label{eq:entropy_variables}
\end{equation}
The entropy flux and entropy potential are
\begin{equation}
    \bm f_\eta=\eta\vel,
    \qquad
    \bm\psi=\rho\vel.
    \label{eq:entropy_flux_potential}
\end{equation}
These definitions are the standard Tadmor entropy structure for the compressible Euler and Navier--Stokes equations \cite{Tadmor2003,Chandrashekar2013}.

For physical viscosity and heat conduction, the viscous flux can be expressed in entropy variables as
\begin{equation}
    \Fv_i=\sum_{j=1}^{3}\bm K_{ij}(\state)\,\partial_{x_j}\entropyvar,
    \qquad
    \bm K_{ij}=\bm K_{ji}^{T},
\end{equation}
where $\bm K=[\bm K_{ij}]$ is symmetric positive semidefinite. Hence,
\begin{equation}
    \mathcal D_v
    =\sum_{i,j=1}^{3}
      (\partial_{x_i}\entropyvar)^T\bm K_{ij}
      (\partial_{x_j}\entropyvar)
    =\nabla\entropyvar:\Fv\geq0.
    \label{eq:viscous_entropy_dissipation}
\end{equation}
The double contraction in the last expression includes the spatial flux direction. Multiplying \cref{eq:compact_ns} by $\entropyvar^T$, applying the entropy compatibility relation, and using the product rule for the viscous term gives the local entropy balance
\begin{equation}
    \partial_t\eta
    +\nabla\cdot\left(\bm f_\eta-\entropyvar^T\Fv\right)
    =-\mathcal D_v+\entropyvar^T\sourceib.
    \label{eq:general_entropy_balance}
\end{equation}
This form retains the viscous entropy flux explicitly rather than absorbing it into the dissipation.

Thus far we have recalled standard results for the entropy structure of the compressible Euler and Navier--Stokes equations. For the immersed boundary method, the entropy contribution of the immersed source is given by contracting the source with the entropy variables 
\begin{align}
    \entropyvar^T\sourceib^{(\theta)}
    &=\frac{\rho}{p}
      \left(\vel\cdot\bm s_{\rho\vel}-s_E^{(\theta)}\right)\nonumber\\
    &=-(1-\theta)\lambda\frac{\rho^2}{p}
      \lvert\vrel\rvert^2.
    \label{eq:entropy_contraction}
\end{align}
Consequently, the complete continuous entropy balance is
\begin{equation}
    \partial_t\eta
    +\nabla\cdot\left(\bm f_\eta-\entropyvar^T\Fv\right)
    =-\mathcal D_v
    -(1-\theta)\lambda\frac{\rho^2}{p}\lvert\vrel\rvert^2.
    \label{eq:continuous_entropy_balance}
\end{equation}
For non-trivial penalisation states, those with $\lambda\lvert\vrel\rvert^2>0$, the immersed contribution is non-positive in the mathematical-entropy convention if and only if $\theta\leq1$.

\subsection{Relative energy and the admissible interval for $\theta$}
\label{sec:relative_energy_source}

\Cref{eq:entropy_contraction} bounds $\theta$ from above through the requirement that the immersed entropy contribution be non-positive. We show here a complementary bound from below, by an energy argument that rests on the fact that the source term in the energy equation is dissipative; it cannot, however, be stated for the total energy. The source exerts a force on the fluid; when the body moves, that force does work, and the work may be positive. The total energy then grows rather than decays. Measuring the energy relative to the solid velocity removes this work and leaves a rate that is non-positive for every prescribed $\vels$.

Write the kinetic and total energies alongside their counterparts relative to the body,
\begin{equation}
    K=\tfrac12\rho\lvert\vel\rvert^2,
    \quad
    \rho E=\rho e+K,
    \qquad
    K_{\mathrm{rel}}=\tfrac12\rho\lvert\vrel\rvert^2,
    \quad
    \rho E_{\mathrm{rel}}=\rho e+K_{\mathrm{rel}},
    \label{eq:energy_definitions}
\end{equation}
the internal energy $\rho e$ of \cref{eq:state} being common to both, since only the kinetic part depends on the frame. Expanding $\lvert\vel-\vels\rvert^2$ gives 
\begin{equation}
    \rho E_{\mathrm{rel}}
    =\rho E-\vels\cdot(\rho\vel)+\tfrac12\rho\lvert\vels\rvert^2 ,
    \label{eq:relative_total_energy}
\end{equation}
which reduces to $\rho E$ when $\vels=\bm0$.

Throughout, $\left.\partial_t(\cdot)\right|_{\mathrm{IB}}$ denotes the contribution of the source alone, with the viscous and convective fluxes omitted. The mass component of the source vanishes, so $\left.\partial_t\rho\right|_{\mathrm{IB}}=0$ and hence $\rho\left.\partial_t\vel\right|_{\mathrm{IB}}=\bm s_{\rho\vel}$. The prescribed $\vels$ is held fixed. 
The source contributions are then
\begin{align}
    \left.\partial_t K\right|_{\mathrm{IB}}
      &=\vel\cdot\bm s_{\rho\vel}
       =-\lambda\rho\,\vel\cdot\vrel ,
    &\left.\partial_t K_{\mathrm{rel}}\right|_{\mathrm{IB}}
      &=\vrel\cdot\bm s_{\rho\vel}
       =-\lambda\rho\lvert\vrel\rvert^2 ,
      \label{eq:kinetic_source}\\
    \left.\partial_t(\rho E)\right|_{\mathrm{IB}}
      &=s_E^{(\theta)} ,
    &\left.\partial_t(\rho E_{\mathrm{rel}})\right|_{\mathrm{IB}}
      &=s_E^{(\theta)}-\vels\cdot\bm s_{\rho\vel}
       =-\theta\lambda\rho\lvert\vrel\rvert^2 ,
      \label{eq:total_energy_source}
\end{align}
the internal energy being common to both frames,
\begin{equation}
    \left.\partial_t(\rho e)\right|_{\mathrm{IB}}
      =s_E^{(\theta)}-\vel\cdot\bm s_{\rho\vel}
       =(1-\theta)\lambda\rho\lvert\vrel\rvert^2 .
    \label{eq:internal_energy_source}
\end{equation}
The relative kinetic rate uses $\partial_t\vrel=\partial_t\vel$, the internal-energy rate follows from $\rho e=\rho E-K$, and the relative total-energy rate from \cref{eq:relative_total_energy}.

\begin{remark}
Notice the contrast between the columns in \cref{eq:kinetic_source} and \cref{eq:total_energy_source} as the left-hand rates contain $\vel$ and $\vels\cdot\vrel$ and are therefore of no definite sign once the body moves, whereas their relative counterparts are sign definite for every prescribed $\vels$. This is the main motivation for writing the energy balance relative to the body.
\end{remark}

The results so far isolate the source term; we now consider the balance for the complete system. Differentiating \cref{eq:relative_total_energy} in time, and keeping separate the terms in which $\vels$ itself is differentiated, gives
\begin{equation}
    \partial_t(\rho E_{\mathrm{rel}})
    =\partial_t(\rho E)
     -\vels\cdot\partial_t(\rho\vel)
     +\tfrac12\lvert\vels\rvert^2\,\partial_t\rho
     -\rho\vrel\cdot\partial_t\vels ,
    \label{eq:rel_time_derivative}
\end{equation}
where the last term collects $-(\rho\vel)\cdot\partial_t\vels+\rho\vels\cdot\partial_t\vels$. The first three terms involve the conservative variables so they are supplied directly by \cref{eq:compact_ns}. Substituting \cref{eq:compact_ns} into those three terms, using $\partial_t\rho=-\nabla\cdot(\rho\vel)$ in the third, and collecting the contributions in which $\vels$ is differentiated gives the complete relative-energy balance
\begin{align}
    \partial_t(\rho E_{\mathrm{rel}})
    +\nabla\cdot\underbrace{\Big(
       \rho E_{\mathrm{rel}}\vel
       +p\vrel
       -\bm\tau\vrel
       +\bm j_h
    \Big)}_{\text{transport}}
    ={}& \underbrace{(\bm\tau-p\bm I):\nabla\vels}_{\text{(a)}}
    \nonumber\\
    &\underbrace{-\rho\vrel\cdot
    \left(\partial_t\vels+(\vel\cdot\nabla)\vels\right)}_{\text{(b)}}
    \underbrace{-\theta\lambda\rho\lvert\vrel\rvert^2}_{\text{(c)}} ,
    \label{eq:full_relative_energy_balance}
\end{align}

Terms (a) and (b) represent power supplied by the prescribed body motion. Both vanish for a stationary body or a constant translational velocity but a deforming, rotating, or accelerating $\vels$ retains them. Term (c), the penalisation, is always non-positive for $\theta\geq0$, independently of the magnitude and direction of $\vels$. Apart from transport and this body work, it is therefore the only mechanism that changes $\rho E_{\mathrm{rel}}$.


Collecting the two conditions gives the central result of this section. Assume $\rho>0$, $p>0$, $\lambda\geq0$, and a prescribed $\vels$, and consider non-trivial penalisation states, for which $\lambda\lvert\vrel\rvert^2>0$. By \cref{eq:total_energy_source} the source is dissipative in relative total energy if and only if $\theta\geq0$, and by \cref{eq:entropy_contraction} it is stable for the convex mathematical entropy if and only if $\theta\leq1$. The two requirements are simultaneously met, for all admissible states, precisely on their intersection,
\begin{equation}
    0\leq\theta\leq1,
    \label{eq:admissible_interval}
\end{equation}
where the source satisfies
\begin{equation}
    \left.\partial_t(\rho E_{\mathrm{rel}})\right|_{\mathrm{IB}}
    =-\theta\lambda\rho\lvert\vrel\rvert^2\leq0,
    \qquad
    \entropyvar^T\sourceib^{(\theta)}
    =-(1-\theta)\lambda\frac{\rho^2}{p}\lvert\vrel\rvert^2\leq0 .
    \label{eq:simultaneous_estimates}
\end{equation}

Each end point is ideal in one budget and degenerate in the other. At $\theta=0$ the source conserves relative total energy and dissipates entropy at the maximal rate. At $\theta=1$ it dissipates relative total energy at full rate and the entropy contraction vanishes, so the source is entropy neutral. A value of $\theta>1$ makes the contraction positive, so the mathematical entropy grows. The symmetric value $\theta=1/2$ places each budget at half its extremal rate.

\section{Split-form DGSEM discretisation}
\label{sec:dgsem_formulation}

This section follows the standard SBP/DGSEM entropy formalism for curvilinear hexahedral elements \cite{FisherCarpenter2013,CarpenterEtAl2014,GassnerWintersKopriva2016,GassnerEtAl2018}. The detailed construction is given for Gauss--Lobatto nodes. The immersed-source result extends to any collocated method with a positive diagonal norm, including Gauss collocation, provided the underlying bulk discretisation already satisfies an entropy estimate.

\subsection{Curvilinear mapping and polynomial approximation}

Let the domain be partitioned into non-overlapping hexahedral elements $\Omega_e$. Each element is the image of the reference cube $\widehat\Omega=[-1,1]^3$ under a smooth mapping
\begin{equation}
    \x=\bm X_e(\bm\xi),
    \qquad
    \bm\xi=(\xi^1,\xi^2,\xi^3).
\end{equation}
Let $J>0$ be the mapping Jacobian and $J\bm a^\alpha$, $\alpha=1,2,3$, the volume-weighted contravariant basis vectors. A physical flux $\bm F=(\bm F_1,\bm F_2,\bm F_3)$ is transformed according to
\begin{equation}
    \widetilde{\bm F}_\alpha
    =\sum_{j=1}^{3}J a_j^\alpha\bm F_j,
    \qquad
    \nabla\cdot\bm F
    =\frac{1}{J}\sum_{\alpha=1}^{3}
      \frac{\partial\widetilde{\bm F}_\alpha}{\partial\xi^\alpha}.
    \label{eq:contravariant_flux}
\end{equation}
The metric terms are constructed so that the discrete metric identities
\begin{equation}
    \sum_{\alpha=1}^{3}D_\alpha(Ja_j^\alpha)=0,
    \qquad j=1,2,3,
    \label{eq:metric_identities}
\end{equation}
are satisfied, ensuring free-stream preservation on curvilinear meshes \cite{Kopriva2006,GassnerWintersKopriva2016}.

On each coordinate direction, let $\{\xi_i\}_{i=0}^{P}$ be the Gauss--Lobatto nodes, $\{\omega_i\}_{i=0}^{P}$ the positive quadrature weights, and $l_i$ the associated Lagrange basis. The nodal approximation is
\begin{equation}
    \state_P(\bm\xi,t)
    =\sum_{i,j,k=0}^{P}\state_{ijk}(t)
      l_i(\xi^1)l_j(\xi^2)l_k(\xi^3).
    \label{eq:nodal_approximation}
\end{equation}
The one-dimensional derivative and norm matrices are
\begin{equation}
    D_{im}=l_m'(\xi_i),
    \qquad
    W=\operatorname{diag}(\omega_0,\ldots,\omega_P),
    \qquad
    Q=WD,
\end{equation}
and satisfy the SBP identity
\begin{equation}
    Q+Q^T=B,
    \qquad
    B=\operatorname{diag}(-1,0,\ldots,0,1).
    \label{eq:sbp_identity}
\end{equation}
The tensor-product discrete inner product is
\begin{equation}
    \inner{a}{b}
    =\sum_{i,j,k=0}^{P}\omega_i\omega_j\omega_k
      a_{ijk}^Tb_{ijk}.
    \label{eq:discrete_inner_product}
\end{equation}

\subsection{Weak form, flux differencing, and viscous discretisation}

The mapped weak form on an element is
\begin{align}
    \inner{J\partial_t\state}{\bm\varphi}
    &+\int_{\partial\widehat\Omega,P}
      \bm\varphi^T\left(
      \widetilde{\bm F}^{a,*}-\widetilde{\bm F}^{v,*}
      \right)\cdot\widehat{\bm n}\,\dd S_\xi
    \nonumber\\
    &-\inner{\widetilde{\Fa}-\widetilde{\Fv}}{\nabla_\xi\bm\varphi}
    =\inner{J\sourceib}{\bm\varphi}.
    \label{eq:dg_weak_form}
\end{align}
The standard nonlinear inviscid volume term is replaced by flux differencing. Let $\bm F_j^{a,\#}(\state_L,\state_R)$ be a symmetric and consistent two-point flux. Along $\xi^1$, the flux-differencing derivative is
\begin{equation}
    \left(\mathcal D_1\bm F^{a,\#}\right)_{ijk}
    =2\sum_{m=0}^{P}D_{im}
      \sum_{l=1}^{3}\avg{Ja_l^{1}}\,
      \bm F_l^{a,\#}(\state_{ijk},\state_{mjk}),
    \label{eq:flux_differencing}
\end{equation}
the inner sum running over Cartesian directions and $\avg{\cdot}$ denoting the arithmetic mean of the metric over the two nodes. Analogous expressions hold in the other two directions. Appropriate Ranocha type fluxes provide representative entropy-conservative volume fluxes \cite{Chandrashekar2013,Ranocha2018,Ranocha2020}.

For the viscous terms, introduce the auxiliary variable $\bm G=\nabla\entropyvar$ and write $\Fv=\Fv(\state,\bm G)$, with $\bm G$ obtained from a companion gradient equation discretised with the same SBP operators. An entropy-stable BR1 discretisation then takes both the viscous flux and the entropy-variable trace as arithmetic averages,
\begin{equation}
    \bm F^{v,*}_j=\avg{\Fv_j},
    \qquad
    \entropyvar^{*}=\avg{\entropyvar},
    \qquad j=1,2,3,
    \label{eq:br1_traces}
\end{equation}
which is entropy stable on Gauss--Lobatto nodes \cite{BassiRebay1997,GassnerEtAl2018}. Other entropy-stable viscous formulations may be used without changing the immersed-source analysis.

\subsection{Collocated immersed source}
\label{sec:nodal_implementation}

At a solution quadrature node $i$, recover the primitive variables from the conservative state:
\begin{equation}
    \vel_i=\frac{(\rho\vel)_i}{\rho_i},
    \qquad
    (\rho e)_i=(\rho E)_i-\frac12\rho_i\lvert\vel_i\rvert^2,
    \qquad
    p_i=(\gamma-1)(\rho e)_i.
    \label{eq:primitive_recovery}
\end{equation}
Then evaluate $\chi_i$, $\vel_{s,i}$, $\lambda_i=\chi_i/\tau_p$, and $\vrel_i=\vel_i-\vel_{s,i}$. The nodal source is
\begin{equation}
    \bm s_{\mathrm{IB},i}^{(\theta)}=
    \begin{bmatrix}
        0\\
        -\lambda_i\rho_i\vrel_i\\
        -\lambda_i\rho_i\left(
        \vel_{s,i}\cdot\vrel_i+\theta\lvert\vrel_i\rvert^2
        \right)
    \end{bmatrix}.
    \label{eq:nodal_source}
\end{equation}
It is evaluated at every tensor-product node and enters the right-hand side of \cref{eq:dg_weak_form} through the same quadrature norm as the time derivative, so that for a diagonal collocation norm $M_i=J_i\omega_i>0$ the semi-discrete residual takes the form
\begin{equation}
    M_i\frac{\dd\state_i}{\dd t}
    =\bm R_{\mathrm{NS},i}+M_i\bm s_{\mathrm{IB},i}^{(\theta)}.
    \label{eq:nodal_residual}
\end{equation}
No embedded-surface numerical flux or cut-cell quadrature is introduced by the volume source. Define the set of penalised nodes in element $e$ as
\begin{equation}
    \mathcal I_e^s
    =\{(i,j,k):\chi_{e,ijk}>0\}.
    \label{eq:penalised_set}
\end{equation}
Only nodes with $\chi_{e,ijk}>0$ are penalised, so the mask may activate one node, an arbitrary proper subset, or all nodes of an element without changing the local source formula. The stability results require admissible nodal states, $\rho_i>0$ and $p_i>0$, and a non-negative mask.

\section{Semi-discrete entropy and energy stability}
\label{sec:semidiscrete}

Beyond the symmetry and consistency assumed in \cref{sec:dgsem_formulation}, entropy conservation requires the inviscid volume flux to satisfy Tadmor's condition \cite{Tadmor2003}, that is,
\begin{equation}
    \jump{\entropyvar}^T\bm F_j^{a,\#}
    =\jump{\psi_j},
    \qquad j=1,2,3.
    \label{eq:tadmor_condition}
\end{equation}

\begin{theorem}[Semi-discrete entropy stability of the coupled scheme]
\label{thm:semidiscrete_entropy}
Assume that the discrete metric identities \cref{eq:metric_identities} hold, the inviscid volume flux satisfies \cref{eq:tadmor_condition}, the interface flux is entropy stable, the viscous and physical-boundary discretisations satisfy their standard entropy estimates, and the nodal states are admissible. If $\theta\leq1$, $\lambda_{e,ijk}\geq0$, and the immersed source is evaluated at the quadrature nodes, then the coupled DGSEM satisfies
\begin{equation}
    \frac{\dd}{\dd t}
    \sum_{e\in\Omega}\inner{J\eta(\state)}{1}
    +\mathcal B_\eta
    =-\mathcal D_h^{v}
     -\mathcal D_h^{\mathrm{int}}
     -\mathcal D_h^{\mathrm{IB}}
    \leq0,
    \label{eq:semidiscrete_entropy_inequality}
\end{equation}
where $\mathcal B_\eta$ contains the prescribed physical-boundary entropy transfer, $\mathcal D_h^{v}\geq0$ is the viscous entropy dissipation, $\mathcal D_h^{\mathrm{int}}\geq0$ is the interface dissipation, and
\begin{equation}
    \mathcal D_h^{\mathrm{IB}}
    =(1-\theta)
    \sum_{e\in\Omega}\sum_{(i,j,k)\in\mathcal I_e^s}
    \omega_i\omega_j\omega_k J_{e,ijk}
    \lambda_{e,ijk}\frac{\rho_{e,ijk}^2}{p_{e,ijk}}
    \lvert\vel_{e,ijk}-\vel_{s,e,ijk}\rvert^2
    \geq0.
    \label{eq:discrete_ib_dissipation}
\end{equation}
\end{theorem}

\begin{proof}
Choose the entropy variables as test functions in \cref{eq:dg_weak_form} and sum over all elements. The temporal term satisfies $\entropyvar^T\partial_t\state=\partial_t\eta$. The SBP identity and the Tadmor condition reduce the inviscid volume contributions to entropy fluxes at element boundaries. Interior entropy-conservative contributions telescope, while the dissipative interface flux produces $-\mathcal D_h^{\mathrm{int}}$. The entropy-stable viscous discretisation produces $-\mathcal D_h^v$. At each collocation node, \cref{eq:entropy_contraction} gives
\begin{equation}
    \entropyvar_{e,ijk}^T\bm s_{\mathrm{IB},e,ijk}^{(\theta)}
    =-(1-\theta)\lambda_{e,ijk}
      \frac{\rho_{e,ijk}^2}{p_{e,ijk}}
      \lvert\vel_{e,ijk}-\vel_{s,e,ijk}\rvert^2.
\end{equation}
Multiplication by the positive nodal norm $\omega_i\omega_j\omega_kJ_{e,ijk}$ and summation gives $-\mathcal D_h^{\mathrm{IB}}$. Combining all contributions proves \cref{eq:semidiscrete_entropy_inequality}.
\end{proof}

\begin{theorem}[Arbitrary penalised subsets and relative-energy dissipation]
\label{thm:semidiscrete_energy}
Let a collocated discretisation possess a positive diagonal norm and an entropy estimate before the immersed source is added, and apply \cref{eq:nodal_source} at any subset of nodes with $\chi_i\in[0,1]$. For $\lambda_{e,ijk}\geq0$ the entropy estimate is preserved whenever $\theta\leq1$, and the source contribution to the discrete relative total energy is
\begin{equation}
    \left.\frac{\dd}{\dd t}
    \sum_{e\in\Omega}\sum_{i,j,k}\omega_i\omega_j\omega_kJ_{e,ijk}
    \rho E_{\mathrm{rel},e,ijk}
    \right|_{\mathrm{IB}}
    =-\theta\sum_{e\in\Omega}\sum_{(i,j,k)\in\mathcal I_e^s}
    \omega_i\omega_j\omega_kJ_{e,ijk}
    \lambda_{e,ijk}\rho_{e,ijk}\lvert\vel_{e,ijk}-\vel_{s,e,ijk}\rvert^2\leq0
    \label{eq:discrete_energy_source}
\end{equation}
whenever $\theta\geq0$, whatever the prescribed $\vels$. Penalised nodes contribute non-positive terms in both budgets and unpenalised nodes contribute zero, so neither result requires the entire cut element to be penalised, and both apply to positive-weight Gauss and Gauss--Lobatto collocation.
\end{theorem}

\begin{proof}
The two source contributions are weighted sums of the pointwise identities \cref{eq:entropy_contraction} and \cref{eq:total_energy_source}, which hold at every collocation node. Every weight and Jacobian is positive, so multiplication by $\omega_i\omega_j\omega_kJ_{e,ijk}$ and summation preserves the sign of each term. Removing any terms from a sum of non-positive contributions cannot change its sign.
\end{proof}

\Cref{thm:semidiscrete_energy} is the source contribution alone. For a stationary body or constant translation the body-work terms of \cref{eq:full_relative_energy_balance} vanish and the remaining bulk contributions are conservative, so the global relative total energy is non-increasing. Both results are modular as any Navier--Stokes discretisation that already satisfies an entropy or energy estimate retains it after addition of the compatible source. Both rest on a positive diagonal norm and therefore require the quadrature and solution nodes to coincide, which they do not in a modal or an over-integrated DG scheme. \Cref{eq:entropy_contraction} holds pointwise, so care is needed wherever the mask, the body velocity, or the source is projected or filtered.

\section{Exact integration of the source}
\label{sec:time}

The rates \cref{eq:kinetic_source,eq:internal_energy_source,eq:total_energy_source} are already restricted to the source, under which $\left.\partial_t\rho\right|_{\mathrm{IB}}=0$, and they can be integrated in closed form rather than approximated by a time-stepping scheme. This section derives that solution and shows that it preserves the continuous estimates for any step size. It also provides a closed-form reference against which a numerical implementation of the source can be checked.

Since $\rho\left.\partial_t\vel\right|_{\mathrm{IB}}=\bm s_{\rho\vel}$ and $\vels$ is held fixed, \cref{eq:momentum_source} gives
\begin{equation}
    \frac{\dd\vrel}{\dd t}=-\lambda\vrel,
    \label{eq:relaxation_ode}
\end{equation}
so, writing $\Delta(\cdot)$ for the change over the step and $a=\lambda\Delta t$,
\begin{equation}
    \vrel^{n+1}=e^{-a}\vrel^n.
    \label{eq:exact_velocity}
\end{equation}
Integrating \cref{eq:internal_energy_source} gives
\begin{equation}
    \Delta(\rho e)
    =\frac{1-\theta}{2}\rho
     \left(1-e^{-2a}\right)\lvert\vrel^n\rvert^2
    \geq0
    \quad\text{for }\theta\leq1 .
    \label{eq:exact_internal_energy}
\end{equation}
The conservative variables are reconstructed as
\begin{equation}
    (\rho\vel)^{n+1}
    =\rho\left(\vels+e^{-a}\vrel^n\right),
    \qquad
    (\rho E)^{n+1}
    =(\rho e)^{n}+\Delta(\rho e)
    +\frac12\rho\left\lvert\vels+e^{-a}\vrel^n\right\rvert^2.
    \label{eq:exact_conservative_update}
\end{equation}
Integrating \cref{eq:total_energy_source} in the same way,
\begin{equation}
    \Delta(\rho E_{\mathrm{rel}})
    =-\frac{\theta}{2}\rho
     \left(1-e^{-2a}\right)\lvert\vrel^n\rvert^2
    \leq0
    \quad\text{for }\theta\geq0 .
    \label{eq:exact_energy}
\end{equation}

The factor $1-e^{-2a}$ is non-negative for every $a=\lambda\Delta t\geq0$, so both signs hold for any step size. Internal energy therefore cannot decrease, so $(\rho e)^n>0$ gives $p^{n+1}>0$, and with $\rho$ constant \cref{eq:mathematical_entropy} integrates to
\begin{equation}
    \Delta\eta
    =-\frac{\rho}{\gamma-1}
     \ln\!\left(1+\frac{\Delta(\rho e)}{(\rho e)^n}\right)\leq0,
    \label{eq:exact_entropy_change}
\end{equation}
which vanishes at $\theta=1$, where the source is entropy neutral.

Integrated exactly, the source therefore leaves the admissible interval \cref{eq:admissible_interval} unchanged over the step, whatever its size. The update assumes that the mask, the relaxation rate and the solid velocity are frozen over the step, and these results hold for the exact update alone. A numerical time integrator applied to the same source need not reproduce them, and a fully discrete analysis lies outside the scope of this work, the estimates of \cref{sec:semidiscrete} remaining semi-discrete.

\section{Numerical experiments}
\label{sec:numerical}

We now verify the results of \cref{sec:continuous,sec:semidiscrete,sec:time} in three complementary configurations. All computations are non-dimensional. Lengths are scaled by a reference length $L_0$, velocities by a reference velocity $u_0$, density by $\rho_0$, pressure by $\rho_0u_0^2$, and time by the convective unit $L_0/u_0$, so that the reference Mach and Reynolds numbers are
\begin{equation}
    M_0=\frac{u_0}{\sqrt{\gamma p_0/\rho_0}},
    \qquad
    \mathrm{Re}=\frac{\rho_0u_0L_0}{\mu},
    \label{eq:nondimensionalisation}
\end{equation}
and the reference pressure and internal energy follow as $p_0=\rho_0u_0^2/(\gamma M_0^2)$ and $(\rho e)_0=p_0/(\gamma-1)$.

The computations use the split-form DGSEM of \cref{sec:dgsem_formulation} on Gauss--Lobatto nodes throughout that satisfy the summation-by-parts property, with the BR1 scheme for the viscous terms. Time integration is performed throughout with an explicit low-storage third-order Runge--Kutta scheme. Every configuration below is periodic in all directions, so no physical-boundary term enters any budget and $\mathcal B_\eta=0$ in \cref{eq:semidiscrete_entropy_inequality} throughout. The implementation uses a binary mask and a penalisation time $\tau_p$ that is uniform within each element, so that $\lambda_{e,ijk}=\chi_{e,ijk}/\tau_p$ is either zero or $1/\tau_p$ at a given node; the theory admits $\chi_{e,ijk}\in[0,1]$ and hence a $\lambda_{e,ijk}$ that varies continuously from node to node. All runs stay inside the admissible interval $0\leq\theta\leq1$, so the necessity of \cref{eq:admissible_interval} rests on the analysis rather than on computation.
The implementation was developed within the HORSES3D framework~\cite{FerrerEtAl2023}, an open-source high-order DGSEM solver
developed at the ETSIAE-UPM School of Aeronautics in Madrid, and available on Github (\url{https://github.com/loganoz/HORSES3D}) with the GPU ported version available in (\url{https://github.com/loganoz/HORSES3D-gpu}).

\subsection{Homogeneous source relaxation}
\label{subsec:relaxation}

This case verifies the closed forms of \cref{sec:time}. It does so independently of the spatial discretisation, isolating the source contribution with no bulk operator acting alongside it. Consider $\rho=1$, $\lambda=1$, $\lvert\vel^0\rvert=1$ and $\vels=\bm0$, and integrate to $t=1$, so that $a=\lambda t=1$.

Evaluated for these parameters, \cref{eq:exact_velocity,eq:exact_internal_energy,eq:exact_energy,eq:exact_entropy_change} give
\begin{align}
    \Delta K &=-\tfrac12\left(1-e^{-2}\right),
    &\Delta(\rho e) &=\tfrac{1-\theta}{2}\left(1-e^{-2}\right),
    \nonumber\\
    \Delta(\rho E) &=-\tfrac{\theta}{2}\left(1-e^{-2}\right),
    &\Delta\eta &=-\frac{1}{\gamma-1}
      \ln\!\left(1+\frac{\Delta(\rho e)}{(\rho e)_0}\right).
    \label{eq:relaxation_exact}
\end{align}

These values are reproduced numerically on a periodic unit box discretised with $2^3$ elements at polynomial order $P=3$. The mask encloses the entire mesh, so $\chi_{e,ijk}=1$ at every node, and the initial state is uniform. Every inviscid and viscous flux therefore vanishes identically and, because all nodes relax at the same rate, the state remains uniform for all time. The discrete system reduces to \cref{eq:relaxation_ode}, and any departure from the closed form is attributable to time-integration error alone. The reference Mach number is $M_0=0.2$, so $(\rho e)_0=44.642857$, and $1000$ steps of size $\Delta t=10^{-3}$ are taken so that $\lambda t=1$ at the final time.

\Cref{tab:relaxation_results} reports the computed values against \cref{eq:relaxation_exact}. Every quantity is recovered with an error of at most $2.4\e{-11}$.

\begin{table}[H]
    \centering
    \begin{tabular}{l r r l}
        \toprule
        Quantity & Computed & Exact & Error\\
        \midrule
        \multicolumn{4}{l}{$\theta=0$: energy conservative (\cref{eq:total_energy_source}); maximal entropy dissipation (\cref{eq:entropy_contraction})}\\
        $\Delta(\rho E)$ & $0.000000000000$  & $0$               & $0$\\
        $\Delta(\rho e)$ & $+0.432332358387$ & $+0.432332358382$ & $5.64\e{-12}$\\
        $\Delta K$ & $-0.432332358387$ & $-0.432332358382$ & $5.64\e{-12}$\\
        $\Delta\eta$ & $-0.024094132728$ & $-0.024094132728$ & $3.13\e{-13}$\\
        \midrule
        \multicolumn{4}{l}{$\theta=1/2$: equal energy dissipation and internal-energy production (\cref{eq:total_energy_source,eq:internal_energy_source})}\\
        $\Delta(\rho E)$ & $-0.216166179203$ & $-0.216166179191$ & $1.19\e{-11}$\\
        $\Delta(\rho e)$ & $+0.216166179185$ & $+0.216166179191$ & $6.21\e{-12}$\\
        $\Delta K$ & $-0.432332358387$ & $-0.432332358382$ & $5.64\e{-12}$\\
        $\Delta\eta$ & $-0.012076092613$ & $-0.012076092613$ & $3.46\e{-13}$\\
        \midrule
        \multicolumn{4}{l}{$\theta=1$: maximal energy dissipation (\cref{eq:total_energy_source}); source entropy neutral (\cref{eq:entropy_contraction})}\\
        $\Delta(\rho E)$ & $-0.432332358405$ & $-0.432332358382$ & $2.36\e{-11}$\\
        $\Delta(\rho e)$ & $-1.80\e{-11}$    & $0$               & $1.80\e{-11}$\\
        $\Delta K$ & $-0.432332358387$ & $-0.432332358382$ & $5.64\e{-12}$\\
        $\Delta\eta$ & $+1.01\e{-12}$    & $0$               & $1.01\e{-12}$\\
        \bottomrule
    \end{tabular}
    \caption{Homogeneous source relaxation at $t=1$ for $\rho=\lambda=\lvert\vel^0\rvert=1$: computed and exact source-induced changes, the exact values following from \cref{eq:relaxation_exact}. Each end-point interpretation is cited on its group row.}
    \label{tab:relaxation_results}
\end{table}

Four entries deserve note. At $\theta=0$ the total energy is unchanged, confirming that the source removes momentum without removing energy. At $\theta=1/2$ the two energy changes are equal and opposite, and each is half the kinetic-energy loss in the row beneath, which is the equal partition of \cref{eq:internal_energy_source,eq:total_energy_source}. The kinetic-energy change is identical across the sweep, as required by the $\theta$-independence of \cref{eq:kinetic_source}. At $\theta=1$ the internal-energy production and the entropy change both vanish to time-integration accuracy, confirming that the upper end point is exactly entropy neutral.

\subsection{Taylor--Green vortex with an immersed sphere}
\label{subsec:tgv}

Having verified the source in isolation, we now couple it to the complete discretisation. The configuration is a Taylor--Green vortex on a triply periodic box containing an immersed sphere. The sphere cuts the background mesh so that the elements at the interface contain penalised and non-penalised nodes, exercising \cref{thm:semidiscrete_energy}. The initial condition is
\begin{equation}
\begin{aligned}
\rho &= \rho_0, \\
u &= u_0 \sin(x/L_0)\cos(y/L_0)\cos(z/L_0), \\
v &= -u_0 \cos(x/L_0)\sin(y/L_0)\cos(z/L_0), \\
w &= 0,
\end{aligned}
\end{equation}
with the corresponding pressure field
\begin{equation}
p =
\frac{\rho_0 u_0^2}{\gamma M_0^2}
+
\frac{\rho_0 u_0^2}{16}
\left[
\left(
\cos(2x/L_0)+\cos(2y/L_0)
\right)
\left(
\cos(2z/L_0)+2
\right)
\right].
\end{equation}
The parameters are $\gamma=1.4$, $p_0=100$, $\mathrm{Re}=100$, $\rho_0=1$, $u_0=1$, and $L_0=1$, so that $M_0=0.0845$, in a domain spanning $[0,2\pi]^3$. The domain is discretised with $20^3$ elements in polynomial order $P=3$, and the sphere has a radius $R=1$ and is centred at $(\pi,\pi,\pi)$. A fixed time step $\Delta t=8\e{-4}$ is used instead of a prescribed CFL number, so the relaxation rate $\lambda=1/\Delta t$ is identical across all runs and in both volume discretisations.

Two global quantities are monitored, the volume averages of total energy and kinetic energy. In the discrete inner product of \cref{eq:discrete_inner_product} these are
\begin{equation}
    \mathcal T_h=\frac{1}{V}\sum_{e\in\Omega}\inner{J\rho E}{1},
    \qquad
    \mathcal K_h=\frac{1}{V}\sum_{e\in\Omega}\inner{J K}{1},
    \label{eq:discrete_functionals}
\end{equation}
with $V$ the volume of the domain. The total energy tests \cref{eq:discrete_energy_source} directly, since the bulk operator is conservative and the source is therefore the only mechanism that can change $\mathcal T_h$.

Additionally, we monitor the kinetic-energy and entropy rate balances. The kinetic-energy balance subtracts the viscous dissipation, pressure work, and source work from $\partial_t K$, while the entropy balance subtracts viscous dissipation from $\partial_t\eta$. As implemented, the two monitors are
\begin{align}
    \mathcal B_h^{K}
      &=\frac{1}{V}\sum_{e\in\Omega}\inner{J\left(
         \vel\cdot\dot{(\rho\vel)}-\tfrac12\lvert\vel\rvert^2\dot\rho
         +\bm\tau:\nabla\vel-p\,\nabla\cdot\vel\right)}{1}
       -\mathcal S_h ,
      \label{eq:kinetic_balance_monitor}\\
    \mathcal B_h^{\eta}
      &=\frac{1}{V}\left(\sum_{e\in\Omega}\inner{J\,\entropyvar^T\dot\state}{1}
       +\mathcal D_h^{v}\right) ,
      \label{eq:entropy_balance_monitor}
\end{align}
with the immersed contribution
\begin{equation}
    \mathcal S_h
    =-\frac{1}{V}\sum_{e\in\Omega}\sum_{i,j,k}
    \omega_i\omega_j\omega_k J_{e,ijk}\,
    \lambda_{e,ijk}\rho_{e,ijk}\,
    \left(\vel_{e,ijk}-\vel_{s,e,ijk}\right)\cdot\vel_{e,ijk} ,
    \label{eq:discrete_kinetic_source}
\end{equation}
the volume average discrete counterpart of \cref{eq:kinetic_source}. The entropy balance \cref{eq:entropy_balance_monitor} deliberately excludes the immersed contribution, so that it can be seen to vanish from the entropy budget at $\theta=1$, as \cref{thm:semidiscrete_entropy} requires. Note that kinetic energy is not one of the stability criteria and that we make no stability statement about it. However, the balance $\mathcal B_h^{K}$ is instead reported to show that a kinetic-energy-preserving volume flux remains kinetic-energy-preserving once the source is added. This further demonstrates the modularity of the source.

Two experiments are performed. The first uses the kinetic-energy-preserving Pirozzoli volume flux \cite{Pirozzoli2010} and reports $\Delta\mathcal T_h$, $\Delta\mathcal K_h$ and $\mathcal B_h^{K}$; the second uses the entropy-conservative Chandrashekar flux \cite{Chandrashekar2013} and reports $\mathcal B_h^{\eta}$, since \cref{thm:semidiscrete_entropy} requires an entropy-conservative volume flux, whereas the relative-energy estimate does not. In both, the volume two-point flux is reused at the element interfaces. No dissipation is added at any interface, and $\theta$ is swept over $0\leq\theta\leq1$. The tables below are reported at $t=10$, after $12\,500$ steps.

\begin{figure}[H]
    \centering
    \includegraphics[trim=41 40 13 28, clip, height=0.46\textwidth]{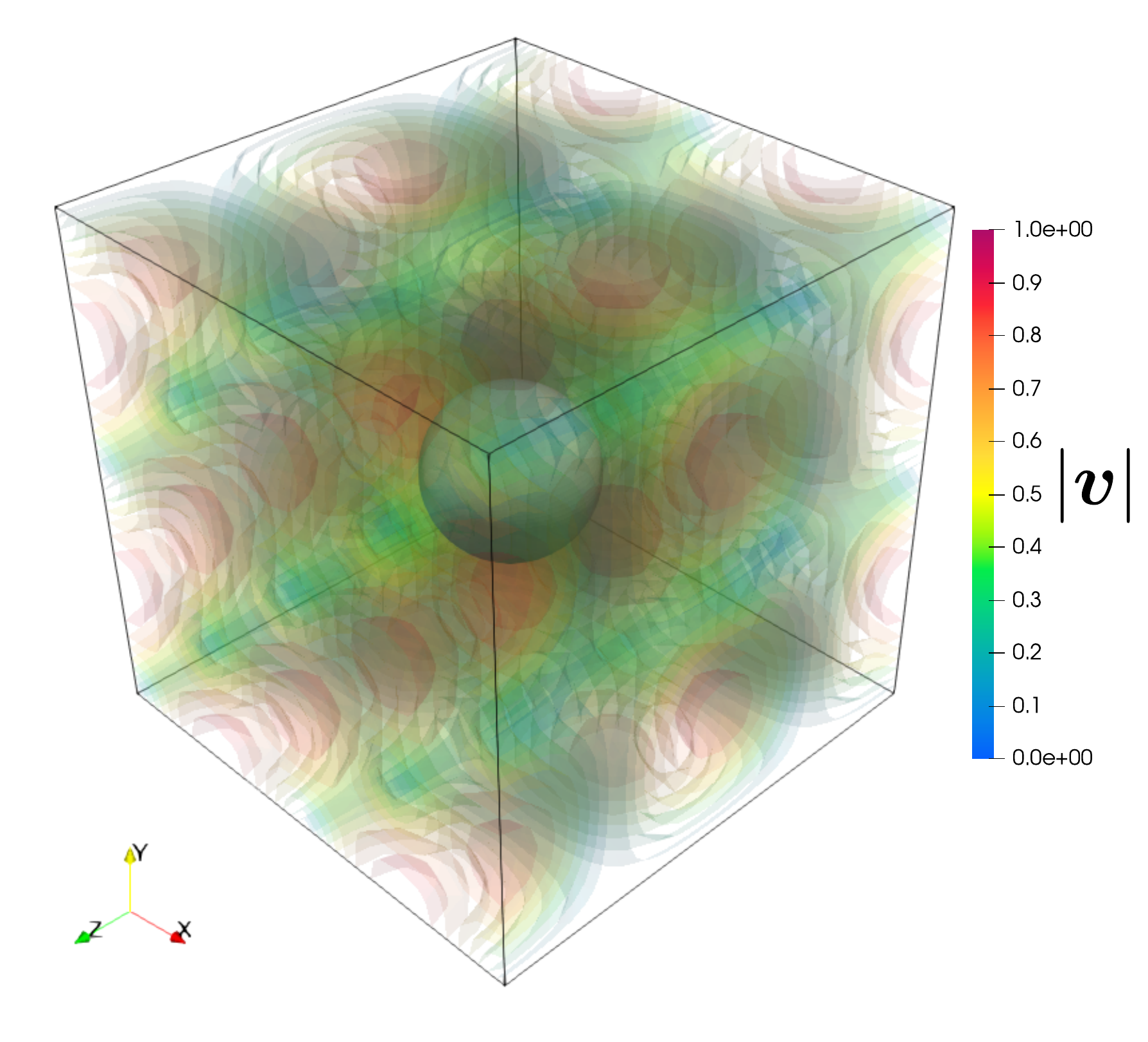}\hfill
    \includegraphics[trim=0 0 99 0, clip, height=0.46\textwidth]{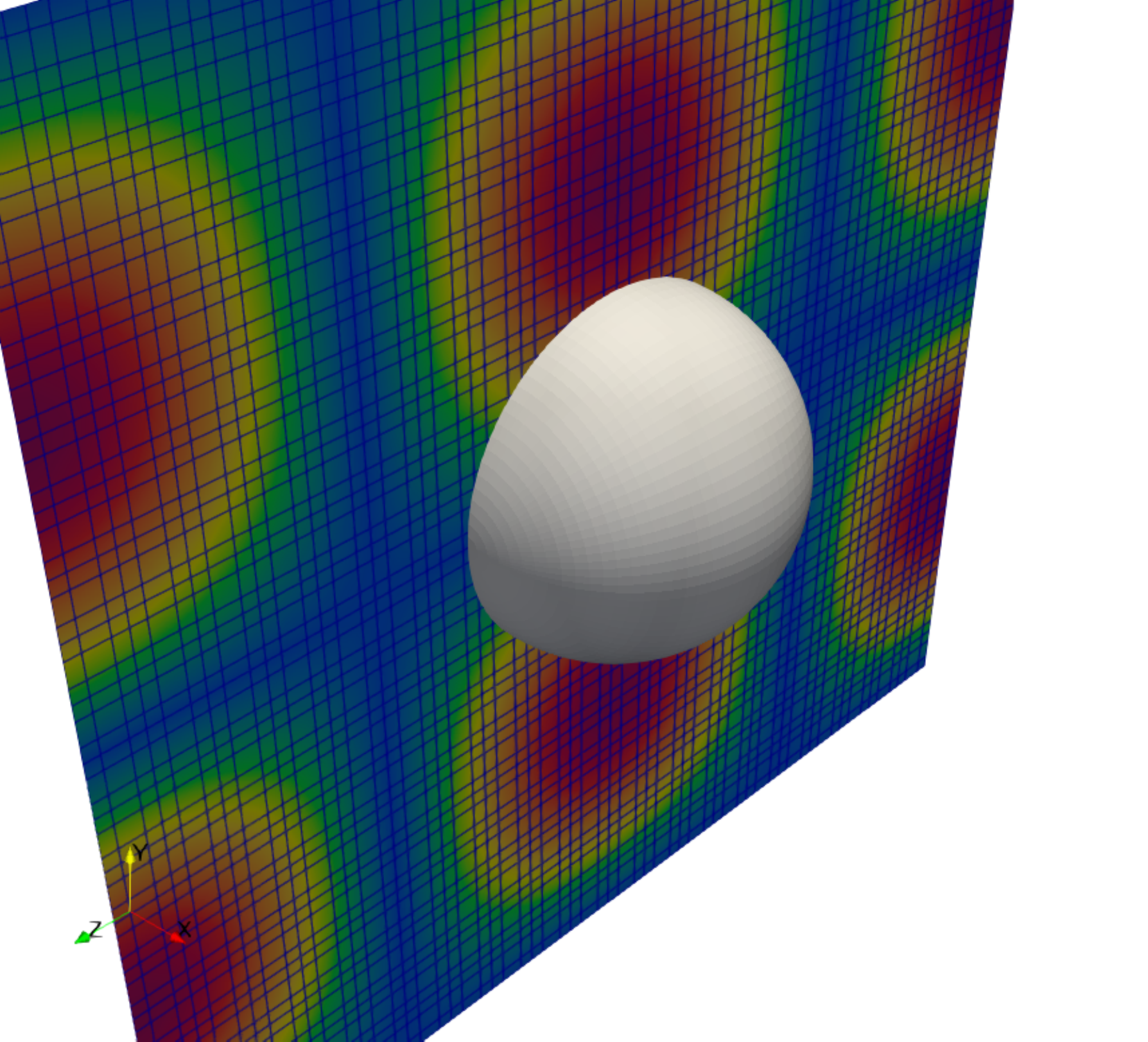}
    \caption{Initial velocity magnitude for the Taylor--Green vortex with the immersed sphere. Left, a transparent rendering of the whole periodic box with the sphere visible at its centre; right, a slice through the domain on which the Gauss--Lobatto collocation nodes are drawn, the immersed body cutting through them so that an element at the interface carries both penalised and unpenalised nodes.}
    \label{fig:tgv_cut}
\end{figure}

\Cref{fig:tgv_cut} shows the initial field, both within the whole periodic box and on a slice through it. The body cuts the background elements, so a cut element carries both penalised and unpenalised nodes, which is what \cref{thm:semidiscrete_energy} requires. The resolution is clearly insufficient even at this polynomial degree; the case is meant to test the budget rather than to produce a converged solution. We note that with the split forms used here the case runs to the final time $t=10$, whereas the same configuration without the split form diverges after $362$ steps at $t=0.29$.

\subsubsection{Energy budgets}

\begin{table}[H]
    \centering
    \begin{tabular}{c c c c}
        \toprule
        $\theta$ & $\Delta\mathcal T_h$ & $\Delta\mathcal K_h$ & $\lvert\mathcal B_h^{K}\rvert$ \\
        \midrule
        $0$   & $-7.11\e{-13}$    & $-0.100360875$ & $9.50\e{-16}$ \\
        $1/2$ & $-1.800335\e{-3}$ & $-0.100360882$ & $8.34\e{-16}$ \\
        $1$   & $-3.600576\e{-3}$ & $-0.100360887$ & $8.47\e{-16}$ \\
        \bottomrule
    \end{tabular}
    \caption{Taylor--Green vortex with an immersed sphere, Pirozzoli volume flux: total-energy and kinetic-energy changes and the discrete kinetic-energy balance at $t=10$.}
    \label{tab:tgv_energy}
\end{table}

\Cref{tab:tgv_energy} confirms four theoretical predictions laid in this work. First, the $\Delta\mathcal T_h$ column shows the total energy is conserved at $\theta=0$. \Cref{thm:semidiscrete_energy} predicts exactly zero there through \cref{eq:discrete_energy_source}, and with $\vels=\bm0$. Second, that same column is linear in $\theta$, as the factor $-\theta$ of \cref{eq:discrete_energy_source} requires, so the two remaining values may be compared directly as $\Delta\mathcal T_h(1)/\Delta\mathcal T_h(1/2)$ is $1.9999480$ against the exact $2$. The agreement is approximate, and we associate the departure with the fact that $\theta$ alters the trajectory of the simulation itself, unlike in \cref{subsec:relaxation}, where each run is compared with its own closed-form solution rather than with another run. Third, the $\lvert\mathcal B_h^{K}\rvert$ column shows that the discrete kinetic-energy budget remains correct once the source is added, the balance of \cref{eq:kinetic_balance_monitor} holding machine zero for every $\theta$. The source removes relative kinetic energy by construction, so the scheme is not kinetic-energy preserving in the presence of the body. What the column shows is that the source contributes exactly the $\mathcal S_h$ of \cref{eq:discrete_kinetic_source} and introduces no further production. Fourth, the $\Delta\mathcal K_h$ column is the same at every $\theta$, as \cref{eq:kinetic_source} requires since only the momentum source enters it.

\begin{figure}[H]
    \centering
    \includegraphics[width=0.8\textwidth]{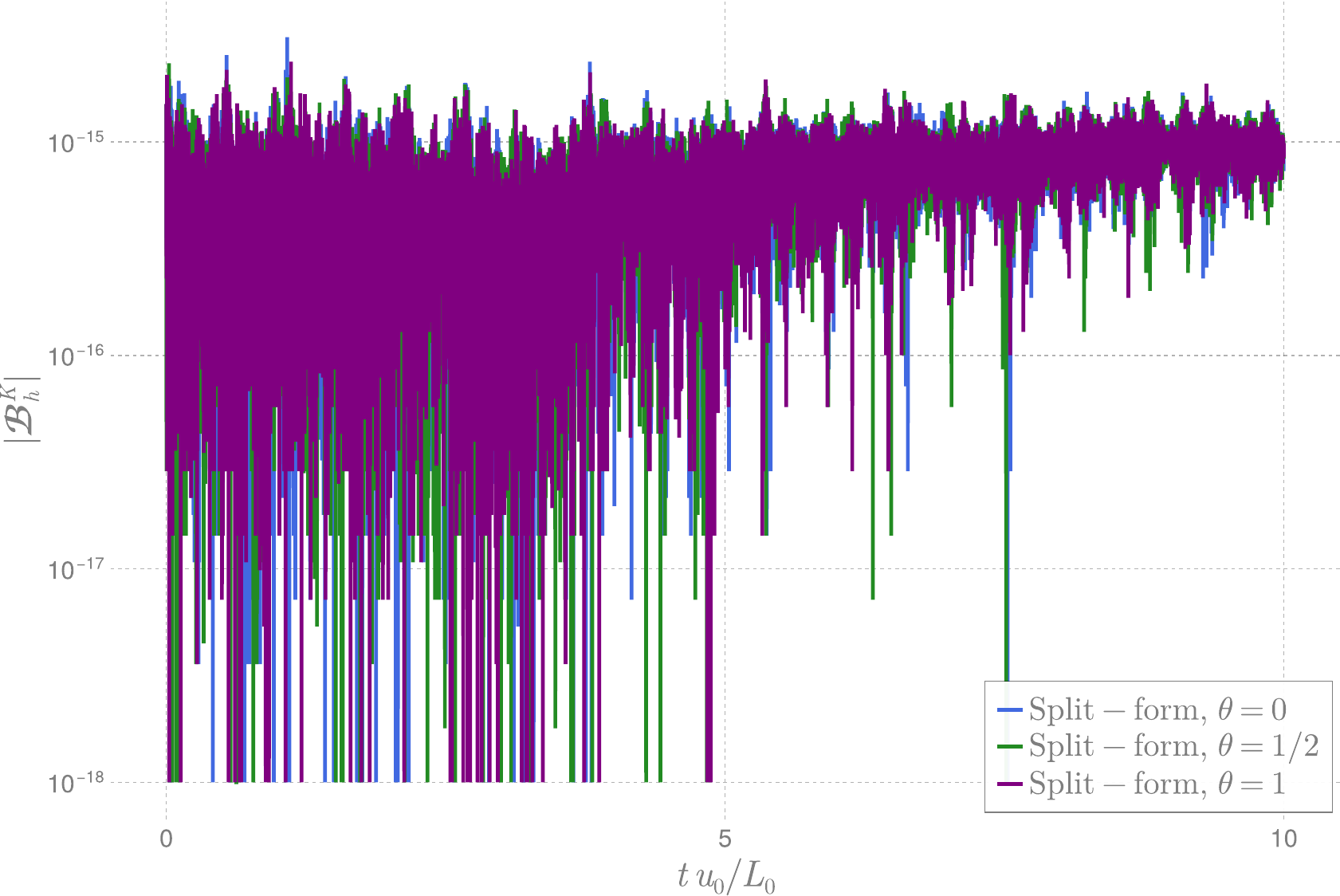}
    \caption{Discrete kinetic-energy balance $\lvert\mathcal B_h^{K}\rvert$ against non-dimensional time for the Taylor--Green vortex with an immersed sphere, over $12\,500$ steps at $\Delta t=8\e{-4}$.}
    \label{fig:kinetic_balance}
\end{figure}

\Cref{fig:kinetic_balance} reports the balance $\mathcal B_h^{K}$ over time and it shows that the balance holds throughout. The kinetic-energy balance of the Pirozzoli flux is therefore unaffected by the immersed source, apart from the source's own contribution $\mathcal S_h$.

\subsubsection{Entropy budget}

\begin{table}[H]
    \centering
    \begin{tabular}{c c c}
        \toprule
        $\theta$ & $\lvert\mathcal B_h^{\eta}\rvert$ & ratio to $\theta=0$ \\
        \midrule
        $0$   & $3.492\e{-8}$  & $1$\\
        $1/2$ & $1.748\e{-8}$  & $0.500633$\\
        $1$   & $1.330\e{-16}$ & $3.810\e{-9}$\\
        \bottomrule
    \end{tabular}
    \caption{Taylor--Green vortex with an immersed sphere, Chandrashekar volume flux: discrete entropy balance at $t=10$, and its ratio to the $\theta=0$ value. The ratio reproduces the factor $1-\theta$ of \cref{eq:discrete_ib_dissipation}, that is $1$, $1/2$ and $0$.}
    \label{tab:tgv_entropy}
\end{table}

\Cref{tab:tgv_entropy} confirms two theoretical predictions concerning the entropy balance. First, the ratio column recovers the $(1-\theta)$ scaling predicted by \cref{eq:discrete_ib_dissipation}. At $\theta=1/2$, the measured value is $0.500633$ against the exact $1/2$. Second, at $\theta=1$, the immersed contribution vanishes identically, leaving the residual to measure only the volume operator. For the entropy-conservative flux, this yields $1.330\e{-16}$, at round-off. The entropy-conservative property of the Chandrashekar flux is therefore preserved by the immersed source with a negative balance as \cref{eq:semidiscrete_entropy_inequality} requires.

\Cref{fig:entropy_balance} shows that the scaling holds throughout rather than only at the final time. The immersed dissipation itself decays, from $5.5\e{-2}$ to $3.5\e{-8}$, because $\mathcal D_h^{\mathrm{IB}}$ is proportional to $\lvert\vrel\rvert^2$ and the penalised region is driven towards rest. Notice how the lines for $\theta=1/2$ and $\theta=0$ track one another exactly, with a shift. The ratio between $\theta=0$ and $\theta=1/2$ stays at $2.00$ through that decay of six orders of magnitude, so the factor $1-\theta$ of \cref{eq:discrete_ib_dissipation} is reproduced continuously. The curve at $\theta=1$ remains below $2.4\e{-16}$ throughout.

\begin{figure}[H]
    \centering
    \includegraphics[width=0.8\textwidth]{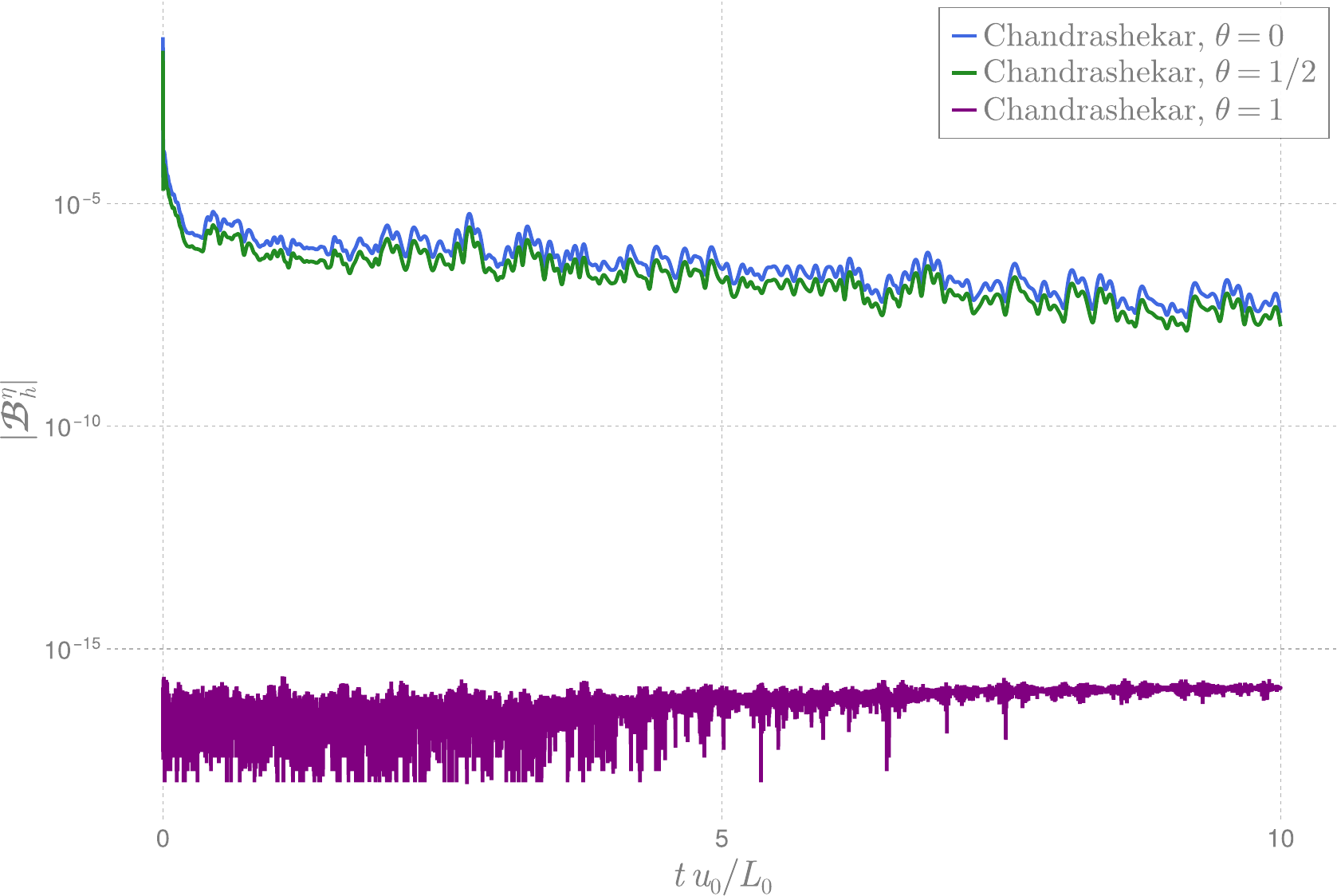}
    \caption{Discrete entropy balance $\lvert\mathcal B_h^{\eta}\rvert$ against non-dimensional time.}
    \label{fig:entropy_balance}
\end{figure}

\subsection{Moving immersed body}
\label{subsec:moving}

Both preceding cases use static geometries (i.e., $\vels=\bm0$), leaving the relative-energy framework untested. This last case tests the implementation of a body in uniform translation, and reports both the total and the relative energy budgets.

A circular cylinder of radius $R=0.5$ translates with constant velocity $\vels=U\bm e_1$ through fluid initially at rest, in a two-dimensional periodic domain $[-4\pi,4\pi]^2$ discretised with $40^2$ elements and polynomial order $P=3$. The parameters are $\gamma=1.4$, $p_0=100$, $\mathrm{Re}=100$, $\rho_0=1$, so that $M_0=0.0845$, and $U$ is set to the reference velocity so that $\lvert\vels\rvert=1$ in non-dimensional terms. The relaxation time is fixed at $\tau_p=10^{-2}$, so that $\lambda=100$, and a fixed step $\Delta t=8\e{-4}$ is used, with $12\,500$ steps taken, so that the cylinder travels ten diameters and stops short of the periodic boundary. As in \cref{subsec:tgv}, two experiments are performed. The first uses the Pirozzoli split form with a Roe interface flux and reports the energy budget, the second the entropy-conservative Chandrashekar flux with the volume flux reused at the interfaces and reports the entropy budget.

\subsubsection{Energy budget}

Constant translation is chosen deliberately. With $\vels$ uniform and steady, $\partial_t\vels=\bm0$ and $\nabla\vels=\bm0$, so both body-work terms on the right of \cref{eq:full_relative_energy_balance} vanish identically and the penalisation is the only source term that remains. We form the relative total energy from the discrete counterpart of \cref{eq:relative_total_energy},
\begin{equation}
    \mathcal T_h^{\mathrm{rel}}
    =\mathcal T_h
     -\frac{1}{V}\sum_{e\in\Omega}\inner{J\,\vels\cdot(\rho\vel)}{1}
     +\frac{1}{2V}\sum_{e\in\Omega}\inner{J\rho\lvert\vels\rvert^2}{1} .
    \label{eq:global_relative_energy}
\end{equation}

\begin{figure}[H]
    \centering
    \includegraphics[width=0.8\textwidth]{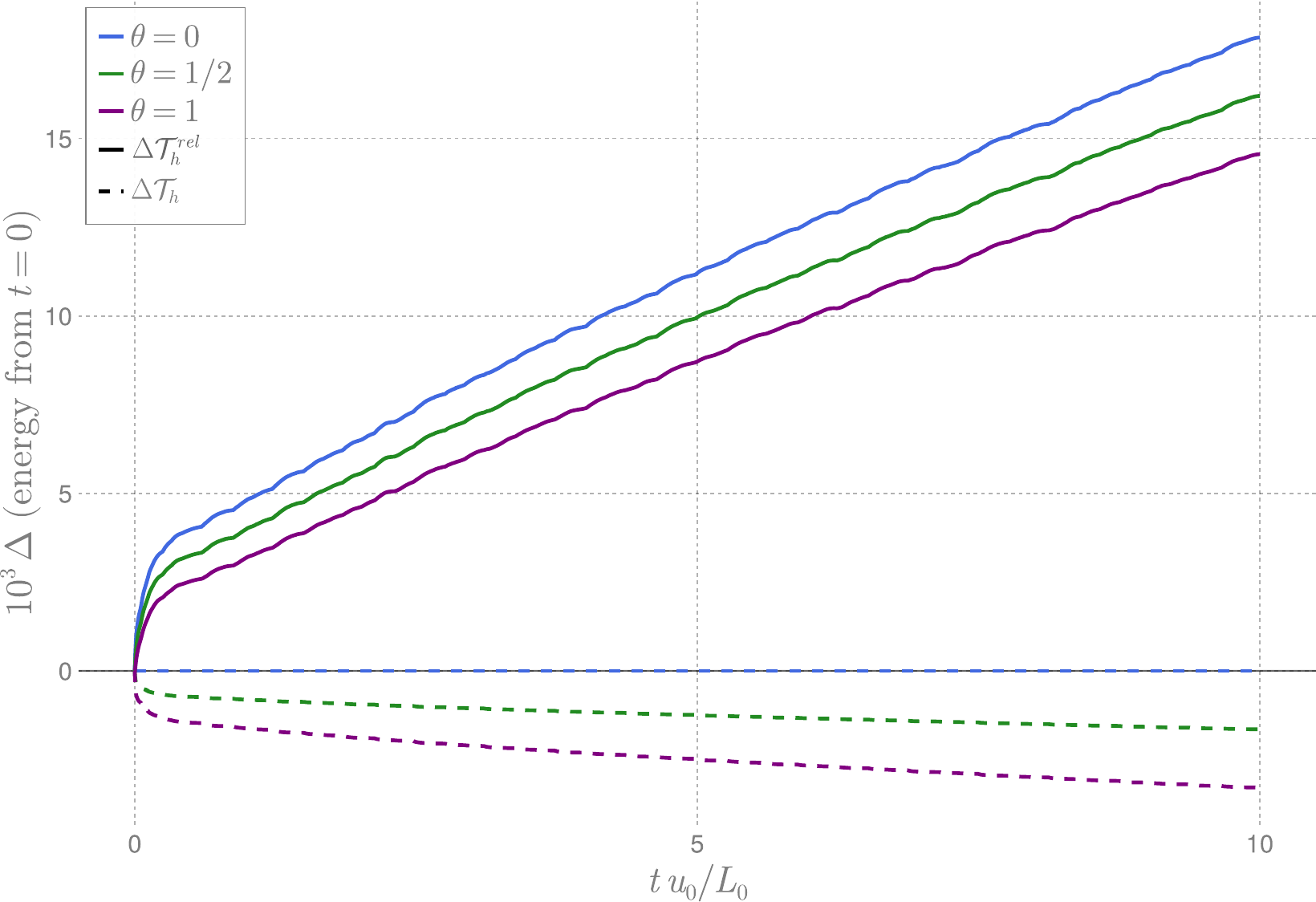}
    \caption{Change in total energy $\mathcal T_h$ (solid) and in relative total energy $\mathcal T_h^{\mathrm{rel}}$ (dashed) against non-dimensional time, over the $12\,500$ steps at $\Delta t=8\e{-4}$. The total energy increases for every admissible $\theta$ while the relative one never does, the difference between them being the work done by the penalisation force at the body velocity. The relative curve at $\theta=0$ remains on zero.}
    \label{fig:moving_energy}
\end{figure}

\Cref{fig:moving_energy} shows both quantities over the whole run. The total-energy curves rise for every admissible $\theta$ while the relative curves are non-increasing, and the gap between the two families is the work $\vels\cdot\bm s_{\rho\vel}$, the term separating the two columns of \cref{eq:total_energy_source}. At $\theta=0$ the relative curve stays on zero to round-off, and the total-energy curve is then the steepest of the three, because none of the kinetic-energy loss leaves the domain. The two remaining relative curves are in the ratio $1.9996$ against the exact $2$, so the factor $-\theta$ of \cref{eq:total_energy_source} holds over the whole run and not merely at one instant. 

The total energy is therefore not a stability diagnostic once the body moves, whereas the relative total energy behaves as \cref{eq:total_energy_source} predicts. Note that this analysis rests on the translation being uniform and steady. For an accelerating body for instance, the term (b) of \cref{eq:full_relative_energy_balance} survives with no sign, so $\mathcal T_h^{\mathrm{rel}}$ may report growth. However, this growth is not produced by the source. It arises from the reference velocity itself changing, and would be present even with $\lambda=0$. The relative energy was introduced to recover the sign-definite estimate on energy for a moving body, and it does so only while $\vels$ is constant.

\subsubsection{Entropy budget}

The entropy estimate carries no condition on $\vels$, so it should survive the body motion unchanged. Repeating the case with the Chandrashekar volume flux reused at the interfaces, and no dissipation added there, isolates the immersed contribution in $\mathcal B_h^{\eta}$ exactly as in \cref{subsec:tgv}.

\begin{table}[H]
    \centering
    \begin{tabular}{c c c}
        \toprule
        $\theta$ & $\lvert\mathcal B_h^{\eta}\rvert$ & ratio to $\theta=0$ \\
        \midrule
        $0$   & $2.467\e{-7}$  & $1$\\
        $1/2$ & $1.235\e{-7}$  & $0.500587$\\
        $1$   & $6.602\e{-17}$ & $2.676\e{-10}$\\
        \bottomrule
    \end{tabular}
    \caption{Translating cylinder, Chandrashekar volume flux: discrete entropy balance at $t=10$, and its ratio to the $\theta=0$ value. The ratio reproduces the factor $1-\theta$ of \cref{eq:discrete_ib_dissipation}, that is $1$, $1/2$ and $0$.}
    \label{tab:moving_entropy}
\end{table}

\Cref{tab:moving_entropy} reproduces for a moving body what \cref{tab:tgv_entropy} shows for a stationary one, the ratio at $\theta=1/2$ being $0.500587$, and the residual at $\theta=1$ falling to round-off, where the immersed contribution vanishes and only the volume operator is left. The balance is negative at every sample, as \cref{eq:semidiscrete_entropy_inequality} requires.

\begin{figure}[H]
    \centering
    \includegraphics[width=0.8\textwidth]{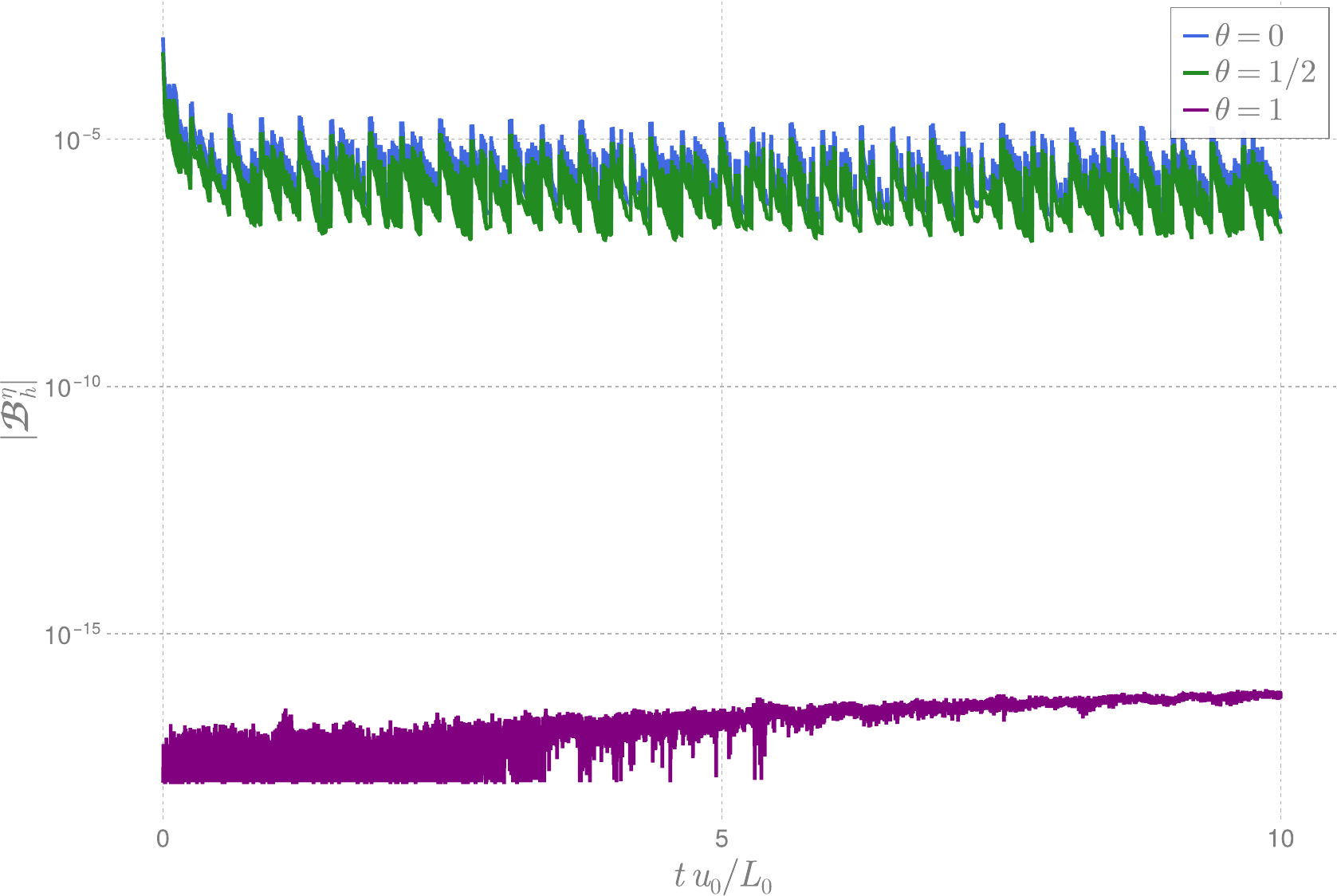}
    \caption{Discrete entropy balance $\lvert\mathcal B_h^{\eta}\rvert$ against non-dimensional time for the translating cylinder.}
    \label{fig:moving_entropy}
\end{figure}

\Cref{fig:moving_entropy} shows that the scaling holds throughout rather than only at the final time. The $\theta=0$ and $\theta=1/2$ curves track one another with a constant shift as the immersed dissipation decays, while the $\theta=1$ curve stays at round-off for the whole run.

\section{Conclusions}
\label{sec:conclusions}

A one-parameter family of compressible Navier--Stokes volume-penalisation sources has been formulated and analysed within the standard entropy-pair and split-form DGSEM framework. Two of its members are already in use: the end point $\theta=1$, as the adiabatic-wall penalisation of \cite{AbgrallEtAl2014}, and the symmetric $\theta=1/2$ \cite{KouEtAl2022VP,KouFerrer2023PoF,FerrerEtAl2023,NunezEtAl2025}. What is established here is the family itself, and the two-sided bound that the relative-energy and entropy requirements place on $\theta$.

For every non-trivial penalisation, the one-parameter momentum--energy source is simultaneously dissipative in total energy measured relative to the solid velocity and stable for the convex mathematical entropy if and only if $0\leq\theta\leq1$. The two requirements bound $\theta$ from opposite sides: the relative-energy estimate from below and the entropy estimate from above. The parameter sets where the kinetic-energy loss goes: a fraction $\theta$ of the relative kinetic energy destroyed by the momentum source leaves the domain, and the remaining $1-\theta$ becomes internal energy. The symmetric value $\theta=1/2$ divides it equally. For a moving body, the prescribed motion does work on the fluid and may raise the total energy. Measured relative to the solid velocity there is no such work, and the penalisation contribution is non-positive.

The split-form DGSEM proof shows that an entropy-stable compressible Navier--Stokes discretisation remains semi-discretely entropy stable after addition of the collocated source. This includes entropy-conservative two-point volume fluxes such as Chandrashekar, with an entropy-stable viscous operator. Kinetic-energy-preserving split forms such as Pirozzoli remain compatible with the relative-energy dissipation of the source as well. The source is therefore modular and can be added to any operator that satisfies the desired stability property, and the resulting discretisation inherits that property.

Neither the entropy nor the relative-energy results require whole-element penalisation. With positive diagonal quadrature, any subset of Gauss or Gauss--Lobatto points in a cut element may be penalised. Active nodes contribute non-positive entropy and relative-energy terms, while inactive nodes contribute zero. Integrated exactly in time, the isolated source preserves relative-energy decay, internal-energy positivity, and physical-entropy production for arbitrary time steps.

The analysis was verified numerically in three configurations. A homogeneous relaxation with every node penalised reproduces the closed-form energy, kinetic-energy, and entropy changes to time-integration accuracy. A Taylor--Green vortex containing a stationary immersed sphere confirms that the total-energy change is linear in $\theta$, that a kinetic-energy-preserving split form retains a machine-zero kinetic-energy balance once the source contribution is subtracted, and that the entropy balance of an entropy-conservative flux falls to round-off at $\theta=1$, where the immersed entropy contribution vanishes. A cylinder in uniform translation shows the total energy increasing for every admissible $\theta$ while the relative total energy never increases and is conserved to round-off at $\theta=0$, where the source contribution vanishes. 
Three scope limitations should be noted. First, $\vels$ is prescribed throughout, and bodies whose motion is coupled to the flow are excluded from the analysis, since the solid energy would then contribute to the overall balance. Second, the global relative-energy statement is restricted to a stationary body or a uniform translation. For an accelerating, rotating, or spatially varying $\vels$ the body-work terms (a) and (b) of \cref{eq:full_relative_energy_balance} remain, carry no definite sign, and $\mathcal T_h^{\mathrm{rel}}$ need not decay, although \cref{thm:semidiscrete_energy} continues to hold for the source contribution at any prescribed $\vels$. Third, the estimates are semi-discrete for the coupled scheme, and exact only for the isolated source. A fully discrete estimate would require a time integrator that preserves the corresponding discrete entropy inequality, whereas the computations presented here employ a standard explicit scheme and therefore satisfy the estimates up to time-integration error.

These results establish a modular route to semi-discretely energy- and entropy-stable immersed discretisations. Provided the bulk discretisation is stable, the source can be added without further modification and the resulting scheme inherits the same stability property.

\section*{Acknowledgements}
E.F. and A.J. acknowledge funding from the European Union (ERC, Off-coustics, project number 101086075). Views and opinions expressed are, however, those of the authors only and do not necessarily reflect those of the European Union or the European Research Council. Neither the European Union nor the granting authority can be held responsible for them.

All authors acknowledge funding from the Grant DeepCFD (Project No. PID2022-137899OB-I00), funded by MICIU/AEI/10.13039/501100011033 and by ERDF, EU.

\section*{Data availability}
Data will be made available on request.

\bibliographystyle{elsarticle-num}
\bibliography{refs}

@article{Peskin2002,
  author  = {Peskin, Charles S.},
  title   = {The immersed boundary method},
  journal = {Acta Numerica},
  volume  = {11},
  pages   = {479--517},
  year    = {2002},
  doi     = {10.1017/S0962492902000077}
}

@article{MittalIaccarino2005,
  author  = {Mittal, Rajat and Iaccarino, Gianluca},
  title   = {Immersed boundary methods},
  journal = {Annual Review of Fluid Mechanics},
  volume  = {37},
  pages   = {239--261},
  year    = {2005},
  doi     = {10.1146/annurev.fluid.37.061903.175743}
}

@article{KimChoi2019,
  author  = {Kim, Woojin and Choi, Haecheon},
  title   = {Immersed boundary methods for fluid--structure interaction: A review},
  journal = {International Journal of Heat and Fluid Flow},
  volume  = {75},
  pages   = {301--309},
  year    = {2019},
  doi     = {10.1016/j.ijheatfluidflow.2019.01.010}
}

@article{Mittal2023,
  author  = {Mittal, Rajat},
  title   = {Origin and evolution of immersed boundary methods in computational fluid dynamics},
  journal = {Physical Review Fluids},
  volume  = {8},
  pages   = {100501},
  year    = {2023},
  doi     = {10.1103/PhysRevFluids.8.100501}
}

@article{Angot1999,
  author  = {Angot, Philippe and Bruneau, Charles-Henri and Fabrie, Pierre},
  title   = {A penalization method to take into account obstacles in incompressible viscous flows},
  journal = {Numerische Mathematik},
  volume  = {81},
  pages   = {497--520},
  year    = {1999},
  doi     = {10.1007/s002110050401}
}

@article{TuPeskin1992,
  author  = {Tu, Cheng and Peskin, Charles S.},
  title   = {Stability and instability in the computation of flows with moving immersed boundaries: A comparison of three methods},
  journal = {SIAM Journal on Scientific and Statistical Computing},
  volume  = {13},
  number  = {6},
  pages   = {1361--1376},
  year    = {1992},
  doi     = {10.1137/0913077}
}

@article{HuaPeskin2022,
  author  = {Hua, Mengjian and Peskin, Charles S.},
  title   = {An analysis of the numerical stability of the immersed boundary method},
  journal = {Journal of Computational Physics},
  volume  = {467},
  pages   = {111435},
  year    = {2022},
  doi     = {10.1016/j.jcp.2022.111435}
}

@article{HuLai2013,
  author  = {Hu, Wen-Feng and Lai, Ming-Chih},
  title   = {An unconditionally energy stable immersed boundary method with application to vesicle dynamics},
  journal = {East Asian Journal on Applied Mathematics},
  volume  = {3},
  number  = {3},
  pages   = {247--262},
  year    = {2013},
  doi     = {10.4208/eajam.250713.150813a}
}

@article{WangJimackWalkley2019,
  author  = {Wang, Yongxing and Jimack, Peter K. and Walkley, Mark A.},
  title   = {Energy analysis for the one-field fictitious domain method for fluid--structure interactions},
  journal = {Applied Numerical Mathematics},
  volume  = {140},
  pages   = {165--182},
  year    = {2019},
  doi     = {10.1016/j.apnum.2019.02.003}
}

@article{BurmanHansbo2012,
  author  = {Burman, Erik and Hansbo, Peter},
  title   = {Fictitious domain finite element methods using cut elements: II. A stabilized Nitsche method},
  journal = {Applied Numerical Mathematics},
  volume  = {62},
  number  = {4},
  pages   = {328--341},
  year    = {2012},
  doi     = {10.1016/j.apnum.2011.01.008}
}

@article{BurmanHansbo2014,
  author  = {Burman, Erik and Hansbo, Peter},
  title   = {Fictitious domain methods using cut elements: III. A stabilized Nitsche method for Stokes' problem},
  journal = {ESAIM: Mathematical Modelling and Numerical Analysis},
  volume  = {48},
  number  = {3},
  pages   = {859--874},
  year    = {2014},
  doi     = {10.1051/m2an/2013123}
}

@article{TaylorWilcoxChan2025,
  author  = {Taylor, Christina G. and Wilcox, Lucas C. and Chan, Jesse},
  title   = {An energy stable high-order cut cell discontinuous Galerkin method with state redistribution for wave propagation},
  journal = {Journal of Computational Physics},
  volume  = {521},
  pages   = {113528},
  year    = {2025},
  doi     = {10.1016/j.jcp.2024.113528}
}

@article{TaylorChan2026,
  author  = {Taylor, Christina G. and Chan, Jesse},
  title   = {An entropy stable high-order discontinuous Galerkin method on cut meshes},
  journal = {Journal of Computational Physics},
  volume  = {548},
  pages   = {114551},
  year    = {2026},
  doi     = {10.1016/j.jcp.2025.114551}
}

@article{KouEtAl2022VP,
  author  = {Kou, Jiaqing and Joshi, Saumitra and Hurtado-de-Mendoza, Aurelio and Puri, Kunal and Hirsch, Charles and Ferrer, Esteban},
  title   = {Immersed boundary method for high-order flux reconstruction based on volume penalization},
  journal = {Journal of Computational Physics},
  volume  = {448},
  pages   = {110721},
  year    = {2022},
  doi     = {10.1016/j.jcp.2021.110721}
}

@article{KouEtAl2022Eigen,
  author  = {Kou, Jiaqing and Hurtado-de-Mendoza, Aurelio and Joshi, Saumitra and Le Clainche, Soledad and Ferrer, Esteban},
  title   = {Eigensolution analysis of immersed boundary method based on volume penalization: Applications to high-order schemes},
  journal = {Journal of Computational Physics},
  volume  = {449},
  pages   = {110817},
  year    = {2022},
  doi     = {10.1016/j.jcp.2021.110817}
}

@article{LlorenteEtAl2023,
  author  = {Llorente, V. J. and Kou, Jiaqing and Valero, Eusebio and Ferrer, Esteban},
  title   = {A modified equation analysis for immersed boundary methods based on volume penalization: Applications to linear advection--diffusion equations and high-order discontinuous Galerkin schemes},
  journal = {Computers \& Fluids},
  volume  = {257},
  pages   = {105869},
  year    = {2023},
  doi     = {10.1016/j.compfluid.2023.105869}
}

@article{KouFerrer2023JCP,
  author  = {Kou, Jiaqing and Ferrer, Esteban},
  title   = {A combined volume penalization/selective frequency damping approach for immersed boundary methods applied to high-order schemes},
  journal = {Journal of Computational Physics},
  volume  = {472},
  pages   = {111678},
  year    = {2023},
  doi     = {10.1016/j.jcp.2022.111678}
}

@article{KouFerrer2023PoF,
  author  = {Kou, Jiaqing and Ferrer, Esteban},
  title   = {A combined volume penalization/selective frequency damping approach for immersed boundary methods: Application to moving geometries},
  journal = {Physics of Fluids},
  volume  = {35},
  pages   = {121702},
  year    = {2023},
  doi     = {10.1063/5.0179779}
}

@article{FerrerEtAl2023,
  author  = {Ferrer, Esteban and Rubio, Gonzalo and Ntoukas, Gerasimos and Laskowski, Wojciech and Mari\~no, Oscar A. and Colombo, Stefano and Mateo-Gab\'in, A. and Narbona, H. and Manrique de Lara, F. and Huergo, David and Manzanero, Juan and Rueda-Ram\'irez, A. M. and Kopriva, David A. and Valero, Eusebio},
  title   = {{HORSES3D}: A high-order discontinuous Galerkin solver for flow simulations and multi-physics applications},
  journal = {Computer Physics Communications},
  volume  = {287},
  pages   = {108700},
  year    = {2023},
  doi     = {10.1016/j.cpc.2023.108700}
}

@article{Colombo2025,
  author  = {Colombo, Stefano and Rubio, Gonzalo and Kou, Jiaqing and Valero, Eusebio and Codina, Ramon and Ferrer, Esteban},
  title   = {A high-order immersed boundary method to approximate flow problems in domains with curved boundaries},
  journal = {Journal of Computational Physics},
  volume  = {528},
  pages   = {113807},
  year    = {2025},
  doi     = {10.1016/j.jcp.2025.113807}
}

@article{Chandrashekar2013,
  author  = {Chandrashekar, Praveen},
  title   = {Kinetic energy preserving and entropy stable finite volume schemes for compressible Euler and Navier--Stokes equations},
  journal = {Communications in Computational Physics},
  volume  = {14},
  number  = {5},
  pages   = {1252--1286},
  year    = {2013},
  doi     = {10.4208/cicp.170712.010313a}
}

@article{CarpenterEtAl2014,
  author  = {Carpenter, Mark H. and Fisher, Travis C. and Nielsen, Eric J. and Frankel, Steven H.},
  title   = {Entropy stable spectral collocation schemes for the Navier--Stokes equations: Discontinuous interfaces},
  journal = {SIAM Journal on Scientific Computing},
  volume  = {36},
  number  = {5},
  pages   = {B835--B867},
  year    = {2014},
  doi     = {10.1137/130932193}
}

@article{GassnerEtAl2018,
  author  = {Gassner, Gregor J. and Winters, Andrew R. and Hindenlang, Florian J. and Kopriva, David A.},
  title   = {The {BR1} scheme is stable for the compressible Navier--Stokes equations},
  journal = {Journal of Scientific Computing},
  volume  = {77},
  number  = {1},
  pages   = {154--200},
  year    = {2018},
  doi     = {10.1007/s10915-018-0702-1}
}

@article{Tadmor2003,
  author  = {Tadmor, Eitan},
  title   = {Entropy stability theory for difference approximations of nonlinear conservation laws and related time-dependent problems},
  journal = {Acta Numerica},
  volume  = {12},
  pages   = {451--512},
  year    = {2003},
  doi     = {10.1017/S0962492902000156}
}

@article{FisherCarpenter2013,
  author  = {Fisher, Travis C. and Carpenter, Mark H.},
  title   = {High-order entropy stable finite difference schemes for nonlinear conservation laws: Finite domains},
  journal = {Journal of Computational Physics},
  volume  = {252},
  pages   = {518--557},
  year    = {2013},
  doi     = {10.1016/j.jcp.2013.06.014}
}

@article{NunezEtAl2025,
  author  = {N{\'u}{\~n}ez, Jonatan and Huergo, David and Lodares, Diego and Shrestha, Suyash and Guerra, Juan and Florenciano, Juan and Ferrer, Esteban and Valero, Eusebio},
  title   = {Implementation of immersed boundaries via volume penalization in the industrial aeronautical computational fluid dynamics solver {CODA}},
  journal = {Engineering with Computers},
  volume  = {41},
  number  = {4},
  pages   = {2571--2592},
  year    = {2025},
  doi     = {10.1007/s00366-025-02119-x}
}

@article{Ranocha2018,
  author  = {Ranocha, Hendrik},
  title   = {Comparison of Some Entropy Conservative Numerical Fluxes for the {Euler} Equations},
  journal = {Journal of Scientific Computing},
  volume  = {76},
  number  = {1},
  pages   = {216--242},
  year    = {2018},
  doi     = {10.1007/s10915-017-0618-1}
}

@incollection{Ranocha2020,
  author    = {Ranocha, Hendrik},
  title     = {Entropy Conserving and Kinetic Energy Preserving Numerical Methods for the {Euler} Equations Using Summation-by-Parts Operators},
  booktitle = {Spectral and High Order Methods for Partial Differential Equations {ICOSAHOM} 2018},
  editor    = {Sherwin, Spencer J. and Moxey, David and Peir{\'o}, Joaquim and Vincent, Peter E. and Schwab, Christoph},
  series    = {Lecture Notes in Computational Science and Engineering},
  volume    = {134},
  pages     = {525--535},
  publisher = {Springer},
  address   = {Cham},
  year      = {2020},
  doi       = {10.1007/978-3-030-39647-3_42}
}

@article{Pirozzoli2010,
  author  = {Pirozzoli, Sergio},
  title   = {Generalized Conservative Approximations of Split Convective Derivative Operators},
  journal = {Journal of Computational Physics},
  volume  = {229},
  number  = {19},
  pages   = {7180--7190},
  year    = {2010},
  doi     = {10.1016/j.jcp.2010.06.006}
}

@article{Kopriva2006,
  author  = {Kopriva, David A.},
  title   = {Metric identities and the discontinuous spectral element method on curvilinear meshes},
  journal = {Journal of Scientific Computing},
  volume  = {26},
  number  = {3},
  pages   = {301--327},
  year    = {2006},
  doi     = {10.1007/s10915-005-9070-8}
}

@article{GassnerWintersKopriva2016,
  author  = {Gassner, Gregor J. and Winters, Andrew R. and Kopriva, David A.},
  title   = {Split form nodal discontinuous Galerkin schemes with summation-by-parts property for the compressible {Euler} equations},
  journal = {Journal of Computational Physics},
  volume  = {327},
  pages   = {39--66},
  year    = {2016},
  doi     = {10.1016/j.jcp.2016.09.013}
}

@article{BassiRebay1997,
  author  = {Bassi, Francesco and Rebay, Stefano},
  title   = {A high-order accurate discontinuous finite element method for the numerical solution of the compressible {Navier--Stokes} equations},
  journal = {Journal of Computational Physics},
  volume  = {131},
  number  = {2},
  pages   = {267--279},
  year    = {1997},
  doi     = {10.1006/jcph.1996.5572}
}

@article{AbgrallEtAl2014,
  author  = {Abgrall, R\'emi and Beaugendre, H\'eloise and Dobrzynski, C\'ecile},
  title   = {An immersed boundary method using unstructured anisotropic mesh adaptation combined with level-sets and penalization techniques},
  journal = {Journal of Computational Physics},
  volume  = {257},
  pages   = {83--101},
  year    = {2014},
  doi     = {10.1016/j.jcp.2013.08.052}
}

@article{LiuVasilyev2007,
  author  = {Liu, Qianlong and Vasilyev, Oleg V.},
  title   = {A {B}rinkman penalization method for compressible flows in complex geometries},
  journal = {Journal of Computational Physics},
  volume  = {227},
  number  = {2},
  pages   = {946--966},
  year    = {2007},
  doi     = {10.1016/j.jcp.2007.07.037}
}

\end{document}